\documentclass[12pt]{amsart}
\usepackage{amssymb,amscd}
\usepackage{verbatim}

\let\frak\mathfrak
\let\Bbb\mathbb

\def\>{\relax\ifmmode\mskip.666667\thinmuskip\relax\else\kern.111111em\fi}
\def\<{\relax\ifmmode\mskip-.333333\thinmuskip\relax\else\kern-.0555556em\fi}
\def\vsk#1>{\vskip#1\baselineskip}
\def\vv#1>{\vadjust{\vsk#1>}\ignorespaces}
\def\vvn#1>{\vadjust{\nobreak\vsk#1>\nobreak}\ignorespaces}

  \let\ssize\scriptstyle
\let\sssize\scriptscriptstyle

\let\Medskip\medskip
\def\medskip{\par\Medskip}
\let\Bigskip\bigskip
\def\bigskip{\par\Bigskip}

\let\Maketitle\maketitle
\def\maketitle{\Maketitle\thispagestyle{empty}\let\maketitle\empty}

\newtheorem{thm}{Theorem}[section]
\newtheorem{cor}[thm]{Corollary}
\newtheorem{lem}[thm]{Lemma}
\newtheorem{prop}[thm]{Proposition}

\numberwithin{equation}{section}

\theoremstyle{definition}
\newtheorem*{rem}{Remark}
\newtheorem*{example}{Example}

\let\mc\mathcal
\let\nc\newcommand

\let\al\alpha

\let\la\lambda
\let\La\Lambda

\let\phi\varphi

\let\Si\Sigma

\let\der\partial

\let\geq\geqslant

\let\leq\leqslant

\let\on\operatorname
\let\bi\bibitem
\let\bs\boldsymbol

\def\C{{\mathbb C}}
\def\Z{{\mathbb Z}}

\def\B{{\mc B}}
\def\Bb{{\mc B}}

\def\F{{\mc F}}

\def\+#1{^{\{#1\}}}

\def\End{\on{End}}
\def\GR{{\on{Gr}_0(H)}}
\def\Gr{{\on{Gr}_{0,0}(H)}}

\def\Wr{\on{Wr}}

\def\sln{\mathfrak{sl}_N}

\def\beq{\begin{equation}}
\def\eeq{\end{equation}}
\def\be{\begin{equation*}}
\def\ee{\end{equation*}}

\nc{\bea}{\begin{eqnarray*}}
\nc{\eea}{\end{eqnarray*}}
\nc{\bean}{\begin{eqnarray}}
\nc{\eean}{\end{eqnarray}}
\nc{\Ref}[1]{{\rm(\ref{#1})}}

\def\g{{\mathfrak g}}
\def\h{{\mathfrak h}}

\let\ga\gamma
\let\Ga\Gamma

\nc{\Il}{{\mc I_{\bs\la}}}
\nc{\bla}{{\bs\la}}
\nc{\Fla}{\F_\bla}
\nc{\tfl}{{T^*\Fla}}
\nc{\GL}{{GL_n(\C)}}
\nc{\GLC}{{GL_n(\C)\times\C^*}}

\let\sd s 

\def\ddk_#1{\kk_{#1}\<\>\frac\der{\der\<\>\kk_{#1}}}

\def\bul{\mathbin{\raise.2ex\hbox{$\sssize\bullet$}}}
\def\intt{\mathchoice
{\mathop{\raise.2ex\rlap{$\,\,\ssize\backslash$}{\intop}}\nolimits}
{\mathop{\raise.3ex\rlap{$\,\sssize\backslash$}{\intop}}\nolimits}
{\mathop{\raise.1ex\rlap{$\sssize\>\backslash$}{\intop}}\nolimits}
{\mathop{\rlap{$\sssize\<\>\backslash$}{\intop}}\nolimits}}

\let\kk q 
\let\cc c

\let\Ko K

\def\GZ/{Gelfand-Zetlin}
\def\KZ/{{\slshape KZ\/}}
\def\qKZ/{{\slshape qKZ\/}}
\def\XXX/{{\slshape XXX\/}}

\nc{\slnl}{{\sln (\lambda)}}
\nc{\PCN}{{   (\C[x])^N   }}
\nc{\di}{\text{Diag}}
\nc{\dio}{\text{Diag}_0}
\nc{\Mm}{{\mc M}}
\nc{\Nn}{{\mc N}}

\nc{\A}{{\mathbb A}}

\nc{\PCr}{{  P  (\C[x])^n   }}

\nc{\Pk}{{(\bs{P}^1)^k}}

\def\hg{{\widehat{\frak{sl}}_N}}
\nc{\N}{{\Bbb N}}

\def\D{{\mc D}}

\nc{\Ll}{{\mc L}}

\nc{\ord}{{\text{ord}\,}}
\nc{\GM}{{\on{Gr}_{mKdV}}}

\begin{document}

\hrule width0pt
\vsk->

\title[Critical points of master functions and integrable hierarchies]
{Critical points of master functions and integrable hierarchies}

\author
[A.\,Varchenko and D.\,Wright]
{ A.\,Varchenko$\>^{\star}$ and D.\,Wright}

\maketitle

\begin{center}
{\it Department of Mathematics, University of North Carolina
at Chapel Hill\\ Chapel Hill, NC 27599-3250, USA\/}
\end{center}

\bigskip

\hfill
{\it To the memory of Andrei Zelevinsky, 1953-2013 }

\bigskip

{\let\thefootnote\relax
\footnotetext{\vsk-.8>\noindent
$^\star$ Corresponding author.
 {\it E-mail address:} anv@email.unc.edu }}

\medskip
\begin{abstract}

We consider the population of critical points
generated from the trivial critical point of the master function with no variables
and  associated with the trivial representation of the affine Lie algebra
$\widehat{\frak{sl}}_N$.
We show that the critical points of this population define
rational solutions of the equations of the mKdV hierarchy associated with $\widehat{\frak{sl}}_N$.

We also construct critical points from suitable $N$-tuples of tau-functions. The construction is based on
a Wronskian identity for tau-functions. In particular, we construct critical points
 from suitable $N$-tuples of Schur polynomials and prove a Wronskian identity for Schur polynomials.

\end{abstract}

\bigskip
\noindent
MSC2010:\ {}\ 37K20 (17B80, 81R10)

\medskip
\noindent
Keywords:  critical points, master functions, mKdV hierarchies,
Miura opers, tau-functions, Schur polynomials, Bethe ansatz, affine Lie algebras


\setcounter{footnote}{0}
\renewcommand{\thefootnote}{\arabic{footnote}}


\section{Introduction}

\subsection{ Master functions}
\label{sec master function}

Let $\g$ be a Kac-Moody algebra with invariant scalar product $(\,,\,)$,
 $\h\subset \g$ the Cartan subalgebra,
  $\al_1,\dots,\al_N$ simple roots,
$\Lambda_1,\dots,\Lambda_n$ dominant integral weights,
 $k_1,\dots, k_N$
nonnegative integers, $k:=k_1+\dots+k_N$.

Consider $\C^n$ with coordinates $z=(z_1,\dots,z_n)$.
Consider $\C^k$ with coordinates $u$ collected into $N$ groups, the $j$-th group consists of $k_j$ variables,
\bea
u=(u^{(1)},\dots,u^{(N)}),
\qquad
u^{(j)} = (u^{(j)}_1,\dots,u^{(j)}_{k_j}).
\eea
The {\it master function} is a multivalued function on $\C^k\times\C^n$,
\bea
&&
\Phi(u,z) = \sum_{a<b} (\La_a,\La_b) \log (z_a-z_b)
- \sum_{a,i,j} (\al_j,\La_a)\log (u^{(j)}_i-z_a) +
\\
&&
+ \sum_{j< j'} \sum_{i,i'} (\al_j,\al_{j'})
\log (u^{(j)}_i-u^{(j')}_{i'})
+  \sum_{j} \sum_{i<i'} (\al_j,\al_{j})
\log (u^{(j)}_i-u^{(j)}_{i'}),
\eea
with singularities at the places where the arguments of the logarithms are equal to zero.
A point in $\C^k\times \C^n$ can be interpreted as a collection of particles in $\C$:\ $z_a, u^{(j)}_i$.
A particle $z_a$ has weight $\La_a$, a particle $u^{(j)}_i$ has weight $-\al_j$.
The particles interact pairwise. The interaction of two particles is proportional to
the scalar product of the weights.
The master function is the "total energy" of the collection of particles.

Notice that all scalar products are integers. So the master function is the logarithm
of a rational function. From a "physical" point of view, all interactions are integer multiples of
a certain unit of measurement. This is important for what will follow.

The variables $u$ are the {\it true} variables, variables $z$ are {\it parameters}.
We will think that the positions of $z$-particles are fixed and $u$-particles can
move.

There are "global" characteristics of this situation,
\bea
I(z,\kappa) = \int e^{\Phi(u,z)/\kappa} A(u,z) du ,
\eea
where $A(u,z)$ is a suitable density function, $\kappa$ a parameter,
and
there are "local" characteristics  -- critical points
of the master function with respect to $u$ variables,
\bea
d_u \Phi(u,z)=0 .
\eea
A critical point is an equilibrium position of the $u$-particles  for fixed positions of the $z$-particles.

Examples of master functions associated with  $\g=\frak{sl}_2$
were considered by Stieltjes and Heine in 19th century, see for example \cite{Sz}.
Master functions we introduced in \cite{SV}
to construct integral representations for solutions of the KZ equations, see also \cite{V1, V2}.

\subsection{ KZ equations and Gaudin model}
\label{sec KZ}
Consider the tensor product $L=\otimes_{a=1}^n L_{\La_a}$ of irreducible
highest weight representations of $\g$ with highest weights $\La_1,\dots,\La_n$.
The KZ equations for an $L$-valued  function $I(z_1,\dots,z_n)$
is a system of differential equations of the form
\bea
\kappa \frac{\der I}{\der z_a} = H_a(z) I,
\qquad
a=1,\dots,n,
\eea
where $\kappa$ is a parameter of the equation and $H_a$ are some linear operators on $L$
defined in terms of the $\g$-action  and called {\it Gaudin Hamiltonians.}
The KZ equations were introduced in CFT, see \cite{KZ}.

In \cite{SV} the KZ equations were identified with a suitable Gauss-Manin connection
and solved in multidimensional hypergeometric integrals.

More precisely, choose nonnegative integers $k_1,\dots,k_N$, consider the associated master
function $\Phi(u,z)$. There is  a rational $L$-valued function $A_{k_1,\dots,k_N}(u,z)$, called
the {\it weight function}, such
that
\bea
I(z,\kappa) = \int_{\gamma(z)} e^{\Phi(u,z)/\kappa} A_{k_1,\dots,k_N}(u,z) du
\eea
is a solution of the KZ equation. Here $\gamma(z)$ is an element of a suitable homology
group. Different  $\gamma$ give different solutions of the KZ equations.

The KZ equations define a flat connection for any $\kappa$. In particular,
the Gaudin Hamiltonians commute, $[H_a(z),H_b(z)]=0$.

The KZ equations are closely related to the  quantum integrable Gaudin model  on $L$, see
\cite{G1, G2}.
In the Gaudin model, one considers the commutative  subalgebra of $ \End(L)$ generated by
the Gaudin Hamiltonians. The subalgebra is called the {\it Bethe algebra}.
 The problem is to diagonalize the operators of the subalgebra.

The integral representations for solutions give a method to diagonalize
the Gaudin Hamiltonians. The method is called the {\it Bethe ansatz}. Let $(u^0,z)$ be a critical point
of the master function with respect to $u$, then the vector $A_{k_1,\dots,k_N}(u^0, z)$ is an eigenvector
of the Gaudin Hamiltonians,
\bea
H_a(z) A_{k_1,\dots,k_N}(u^0,z) = \frac {\der \Phi}{\der z_a}(u^0,z) A_{k_1,\dots,k_N}(u^0,z), \qquad a=1,\dots, n,
\eea
with eigenvalues described by the master function, see \cite{BF, RV, V3}. The natural (Shapovalov) norm of the Bethe vector is given
by the Hessian of the master function at the critical point,
\bea
S(A_{k_1,\dots,k_N}(u^0,z),A_{k_1,\dots,k_N}(u^0,z)) = \text{Hess}_u\, \Phi(u^0,z),
\eea
see \cite{MV3, V3, V4}.

In a reasonable way one can identify the algebra of functions on the critical set
with the Bethe algebra, generated by the Gaudin Hamiltonians, see \cite{MTV1, MTV2}.

The situation is analogous to the situation in quantum cohomology and classical mirror symmetry
in Givental's spirit. There one has the quantum differential equation depending on a parameter $\kappa$,
which is a flat connection for every value of $\kappa$. As $\kappa \to 0$ the asymptotics of the quantum differential equation
are described by idempotents of the quantum cohomology algebra.
The solutions of the quantum differential equation
can be represented by oscillatory integrals. In a reasonable way
the quantum cohomology algebra can be identified with the algebra of functions on the critical set of the
phase function of the oscillatory integral.

\subsection{ Generation of new critical points}
\label{sec Gen new cr points}

Having a critical point $(u,z)$ of a master function one can generate new critical points
by changing $u$ and not changing $z$.

Consider the master function $\Phi$ associated with $\g$, $\Lambda_1,\dots,\La_n$, $k_1,\dots,k_N$.
Assume that $(u,z)$ is a critical point of $\Phi$ with respect to $u$ variables,
$u=(u^{(1)},\dots,u^{(N)}), \ u^{(j)} = (u^{(j)}_1,\dots,u^{(j)}_{k_j}),\ j=1,\dots,N.$

\begin{thm}[\cite{ScV, MV1}]
\label{thm ScV}
 Choose any one group  $u^{(j)}$ of these numbers (particles).
Then there exists a unique algebraic one-parameter deformation $u^{(j)}(c),\ c\in \Bbb P^1$, of that group, such that
$(u^{(1)},\dots,u^{(j)}(c),\dots, u^{(N)},z)$ is a critical point.
For exactly one value $c\in \Bbb P^1$, the number of particles in the group
$u^{(j)}(c)$ drops. For that $c$ some particles of the group disappear
at infinity.
\end{thm}

All points of this one-parameter family but one are critical points of the same master function,
and the one with the smaller number of particles is a critical point of the master function
with a smaller number of variables.

This theorem gives us a way to generate critical points. Starting with a critical point
we can generate $N$ one-parameter families of critical points. Then we may apply the same procedure
to each of the obtained critical points and so on. The set of all critical points obtained from a given one
is called the {\it population} of critical points \cite{MV1}.
\medskip

{\bf Question:}\ What does a population look like?

\begin{thm}
If $\g$ is a simple Lie algebra, then every population is isomorphic to the flag variety of the Langlands
dual Lie algebra $\g^L$.
\end{thm}

This theorem was proved in \cite{MV1} for the Lie algebras of type A, B, C and extended to other simple Lie algebras
in \cite{MV2, F}.

For example, for $\g=\frak{sl}_2$ each critical point generates a projective line $\Bbb P^1$ of critical points.
For $\g=\frak{sl}_3$ each population is isomorphic to the variety $\F(\C^3)$
of complete flags in $\C^3$.
The steps of the generation procedure correspond to the building of $\F(\C^3)$ from Schubert
cells labeled by elements of the symmetric group $\Sigma_3$.
We start with the  zero-dimensional cell $C_{id}$ corresponding to the initial critical point.
Then we add two one-dimensional cells $C_{s_1}$,
$C_{s_2}$. Then add two two-dimensional cells $C_{s_1s_2}$, $C_{s_2s_1}$,
and finally add one three-dimensional cell $C_{s_2s_1s_2}=C_{s_1s_2s_1}$.

\medskip
{\bf Question:}  What is the number of populations
originated from critical points of a given master function?
\medskip

Here is an informal answer:
{\it The number of populations originated from critical points  of the function $\Phi(\,\cdot\,,z)$, associated
with $\La_1,\dots,\La_n, k_1,\dots,k_N$, does not depend on $z$ and equals the multiplicity
of the representation $L_{\La_\infty}$ in $\otimes_{a=1}^n L_{\La_a}$, where }
$\La_\infty$ {\it is the highest weight in the orbit of the weight  $ \sum_{a=1}^n\La_a - \sum_{j=1}^Nk_j\al_j $
under the shifted action of the Weyl group}.

A precise statement for $\g=\frak{sl}_N$ can be found in \cite{MV1}, \cite{MTV3}. For other simple Lie algebras  a proof
of  this statement is not written yet.
This statement on the number of populations can be considered as a relation between linear algebra (the multiplicity of a representation)
and  geometry (the number of populations).

\medskip
We have observed above that master functions
are related to interesting  objects and have nice properties.
Now we will explain how the critical points of master functions  are related to integrable hierarchies.
We expect that these relations are more general than the results that are reported in this paper, see
the statement in  Section \ref{MA}.

\subsection{ Generation for $\hg$  }
\label{sec gene for gh}

In this paper we consider the case $\g=\hg$, $n=0$. Then the master function does not depend on $z$ and is a function
of $
u=(u^{(1)},\dots,u^{(N)}),
\
u^{(j)} = (u^{(j)}_1,\dots,u^{(j)}_{k_j})
$
only,
\bean
\label{Master}
\phantom{aaaaaa}
\Phi(u) =
2 \sum_{j=1}^N  \sum_{i<i'}
\log (u^{(j)}_i-u^{(j)}_{i'})
- \sum_{j=1}^{N-1} \sum_{i,i'}\log (u^{(j)}_i-u^{(j+1)}_{i'}) - \sum_{i,i'}\log (u^{(N)}_i-u^{(1)}_{i'}).
\eean
We may also say that this master function corresponds to the case of an arbitrary $n$ and $\La_1=\dots=\La_n=0$.

Given $u$, introduce an $N$-tuple of polynomials of a new variable $x$,
\bea
y=(y_1(x),\dots,y_N(x)),
\qquad
y^j(x) = \prod_{i=1}^{k_j}(x-u^{(j)}_i).
\eea
For functions $f(x),g(x)$, denote
\bea
\Wr(f,g) = f(x)g'(x)-f'(x)g(x)
\eea
the Wronskian determinant.
An $N$-tuple of polynomials $y=(y_1(x),\dots,y_N(x))$ is called {\it generic} if each polynomial has no multiple roots
and for any $i$ the polynomials $y_i$ and $y_{i+1}$ have no common roots. Here $y_{N+1}=y_1$.

\begin{thm} [\cite{MV1}]
\label{thm a}

 A generic $N$-tuple $y$ represents a critical point if and only if for any $j=1,\dots,N$,
there exists a polynomial $\tilde y_j(x)$ satisfying
\bean
\label{Wr eqn}
\Wr(y_j, \tilde y_j)= y_{j-1}y_{j+1}.
\eean
For $j=1$, this is $\Wr(y_1, \tilde y_1)= y_{N}y_{2}$
and for  $j=N$, this is $\Wr(y_N, \tilde y_N)= y_{N-1}y_{1}$.

\end{thm}

Equation \Ref{Wr eqn} is a first order inhomogeneous differential equation with respect to $\tilde y_j$.
Its solutions are
\bea
\tilde y_j = y_j\int \frac{y_{j-1}y_{j+1}}{y_j^2} dx + c y_j,
\eea
where $c$ is any number. For arbitrary polynomials $y_{j-1}, y_{j+1}, y_j$, the integral will
contain logarithms. The integral is a rational function if the residues of the ratio
 are all equal to zero. The requirement that the residues are all equal
to zero are exactly the critical point equations
for the master function.

\begin{thm}
[\cite{MV1}]
\label{tHm MV}
If a generic $N$-tuple $y$ represents a critical point, then for almost all $c\in \Bbb P^1$,
the $N$-tuple $(y_1,\dots, \tilde y_j,\dots,y_N)$ represents a critical point. The set of exceptional $c$ is finite and consists
of all $c$ for which the $N$-tuple $y$ is not generic.
\end{thm}

We get a one parameter family of critical points parameterized by $c\in \Bbb
P^1$. For exactly one value
$c$ the degree of the $j$-th polynomial drops,
\bea
\tilde k_j + k_j -1 = k_{j-1}+k_{j+1} ,
\eea
where $k_j, \tilde k_j$ are possible degrees of the $j$-th polynomial.

Theorem \ref{tHm MV} describes the generation procedure of critical points for the master function
\Ref{Master}, this generation  was described  in general terms in Theorem \ref{thm ScV}.

Let us choose a starting critical point to generate a population of critical points.
Namely, let us choose the starting $N$-tuple to be
\bea
y^\emptyset = (1,\dots,1).
\eea
Each of the polynomials of the tuple does not have roots. The tuple $y^\emptyset$ represents
the critical point of the master function with no variables. Choose a sequence of integers
$J=(j_1,\dots,j_m),$ $ j_a\in \{1,\dots,N\}$, and apply to the tuple $y^\emptyset$
the sequence of generations  at
the indices of $J$. As a result we will obtain an $m$-parameter family
of $N$-tuples of polynomials
\bea
y^J(x,c)\,=\,(y_1(x,c),\dots,y_N(x,c)),
\eea
 depending on $m$ integration
constants $c=(c_1,\dots,c_m)$.

\begin{example}
 Let $N=3$, $J=(1,2,3,1,2,3,...)$. Then
$y^\emptyset=(1,1,1)$ is transformed to
\bea
(x+c_1,1,1),
\eea
 then to
\bea
(x+c_1, \frac{(x+c_1)^2}2 + c_2, 1),
\eea
then to
\bea
(x+c_1, \frac{(x+c_1)^2}2 + c_2, \frac{(x+c_1)^4}8 + \frac{(x+c_1)^2}2 c_2 + c_3),
\eea
and so on.

The triple $(1,1,1)$ represents the critical point of the master function with no variables.
The triple $(x+c_1,1,1)$ represents critical points of the master function of one variable,
namely, the constant function $\Phi : u^{(1)}_1\mapsto 1$. All points of the line are critical points.
The triple $(x+c_1, \frac{(x+c_1)^2}2 + c_2, 1)$ represents the critical points of the master function
\bea
\Phi = \log \left(\frac{(u^{(2)}_1-u^{(2)}_2)^2} {(u^{(2)}_1-u^{(1)}_1)(u^{(2)}_2-u^{(1)}_1)}
\right).
\eea
\end{example}

\begin{rem}  The Bethe ansatz applied to master functions \Ref{Master} allows us to construct eigenvectors of
$\hg$  Gaudin Hamiltonians $H_a(z)$  acting on a tensor power $M_0^{\otimes n}$ of the $\hg$ Verma modules with zero highest weight.

\end{rem}

\subsection{All critical points of master functions
\Ref{Master} lie in the population generated from $y^\emptyset$}

In this paper we shall study  the population  generated  from $y^\emptyset$ of critical points of the master functions
\Ref{Master}. The following theorem
says that the critical points of this population exhaust all critical points of the considered master functions.

\begin{thm}
[\cite{MV4}]
\label{thm last}

 If a tuple $(y_1,\dots,y_N)$ represents a critical point of the master function
\Ref{Master} for some parameters $k_1,\dots,k_N$, then $(y_1,\dots,y_N)$ is a point
of the population of tuples
 generated  from $y^\emptyset$.

\end{thm}

We will not use this theorem in this paper.


\subsection{ Miura opers and mKdV hierarchy}

An $N\times N$ {\it Miura oper} is a differential operator of the form
\bean
\label{mi oper}
\Ll =\frac d{dx} + \Lambda + V,
\eean
where  $V = \text{diag}\,(f_1(x),\dots,f_N(x))$,
\bean
\label{LAM}
\Lambda = \begin{pmatrix}
0 & 0& \dots & 0 & \la
\\
1 & 0 &\dots & 0 & 0
\\
0 & 1 &\dots & 0 & 0
\\
\dots & \dots  &\dots & \dots & \dots
\\
\dots & \dots  &\dots & \dots & \dots
\\
0 & 0 &\dots & 1 & 0
\end{pmatrix},
\eean
and
 $\lambda$ is a parameter of the differential operator.
Drinfeld and Sokolov in 1984 in \cite{DS} considered the space  of Miura opers
and defined on it an infinite sequence of commuting flows $\der_{t_r}, r=1,2,\dots$,
called the {\it mKdV hierarchy} of $\widehat{\frak{sl}}_N$-type.

More precisely, for a Miura oper $\Ll$  there exists $T = 1 + \sum_{i < 0}T_i(x) \Lambda^i$,
 where $T_i(x)$ are diagonal matrices, such that
 $T^{-1} \Ll T$ has the form
$
 \frac d{dx} + \Lambda + \sum_{i \leq 0} b_i(x) \Lambda^i,
 $  where $b_i(x)$ are scalar functions, see \cite{DS}.
Then for $r\in\Bbb N$, the differential equation
\bean
\label{Mkdv}
\frac{\partial \Ll}{\partial t_r}  = [\Ll,\left(T \Lambda^r T^{-1} \right)^+]
\eean
 is called  the $r$-th  mKdV equation.
 The notation $()^+$ has the following meaning:
 given $M=\sum_{i\in\Bbb Z}d_i\Lambda^i$ with diagonal matrices $d_i$, then
 $M^+=\sum_{i\geq 0}d_i\Lambda^i$.

Equation \Ref{Mkdv} defines vector fields $\der_{t_r}$ on the space of Miura opers, the vector fields commute.

\subsection{Main result}
\label{MA}

To every $N$-tuple $y$ of polynomials we assign the Miura oper $\Ll$ with
\bea
V= \text{diag} \left(\log'\left(\frac{y_1}{y_N}\right),\log'\left(\frac{y_2}{y_1}\right),
\dots, \log'\left(\frac{y_N}{y_{N-1}}\right)\right).
\eea
 For a sequence $J=(j_1,\dots,j_m)$, we consider the corresponding $m$-parameter family
 of critical points of the master function, represented by the family
 of $N$-tuples of polynomials $y^J(x,c)$. Let $\Ll^J(c)$ be the corresponding
 family of Miura opers.

\medskip

In this paper we show that {\it this family  $\Ll^J(c)$
of Miura opers is invariant under every flow of the mKdV hierarchy and is
point-wise fixed by every flow $\der_{t_r}$ with $r> 2m$.}

This statement is proved under the assumption
that $J$ is a degree increasing sequence, see the definition  in
Section \ref{sec degree transf}. This is a technical assumption which simplifies the exposition.

\medskip

Our starting point was the famous paper by Adler and Moser (1978) \cite{AM}  where this theorem was proved for $N=2$.

\subsection{ Identities for  Schur polynomials}

In this paper we describe two proofs of the main result. The first proof is straightforward: for any flow of the
hierarchy we deform the constants of integration of the generation procedure
to move the oper in the direction of the flow. This proof is similar to the proof in \cite{AM}.
The second proof uses tau-functions and, in particular, Schur polynomials. It is based on a Wronskian identity
for tau-functions and, in particular, for Schur polynomials.

Schur polynomials $F_\la(t_1,t)$ are labeled by partitions $\la=(\la_0\geq\la_1\geq\dots\geq\la_n\geq 0)$
and are polynomials in $t_1$ and $t=(t_2,t_3,\dots)$.
Let us define polynomials $h_i(t_1,t)$, $i=0,1,\dots$, by the relation
$
\text{exp}\,(-\sum_{j=1}^\infty t_jz^j)\,=\, \sum_{i=0}^\infty h_iz^i ,
$
and set
$
F_\la\, =\, \text{det}_{i,j=0}^n\,(h_{\la_i-i+j}).
$
In this paper we prove that certain 4-tuples of Schur polynomials satisfy the Wronskian identity
\bean
\label{ONE}
\Wr_{t_1}(F_{\la^1}, F_{\la^2})\ =\ F_{\la^3}F_{\la^4},
\eean
where $\Wr_{t_1}$ denotes the Wronskian determinant with respect to $t_1$.
For example,
\\
$W_{t_1}(F_{(2,1)},F_{(0)}) = F_{(1)}F_{(1)}$ and
$\Wr_{t_1}(F_{(4,2,1)}, F_{(2,2,1)}) = F_{(3,2,2,1)}F_{(2,1)}.$

We also prove that for any $N$, certain $N$-tuples of Schur polynomials
\bean
\label{TWO}
y(t_1,t)\, =\, \left(F_{\la^1}(t_1,t),\dots, F_{\la^N}(t_1,t)\right)
\eean
are such that for a fixed generic  $t$, the tuple $y(x,t)$ represents a critical point of the master
function \Ref{Master}.
In other words, we show that for a fixed generic  $t$, the roots of the polynomials
$(F_{\la^1}(t_1,t),\dots, F_{\la^N}(t_1,t))$ with respect to $t_1$ satisfy the critical point equations
\Ref{Bethe eqn 1}.  For example, for $N=3$ the triple
$
(F_{(1,1)}, F_{(2,1,1)}, F_{(1)})
$
 is such a triple.

\subsection{Exposition of the material}

In Section \ref{sec kdv and mkdv} we remind the construction of the mKdV and KdV hierarchies from \cite{DS}.
In Section \ref{CP} we remind the construction of the generation of critical points from \cite{MV1}.
In Section \ref{MOCP} we assign a Miura oper to a critical point as in \cite{MV2} and discuss properties of this assignment.
In Section \ref{Vector fields} we formulate and prove one of the main results of the paper, Corollary \ref{cor Main}.
Corollary \ref{cor Main}  says that
the family $\Ll^J$ of Miura opers, associated with the generation of critical points in the direction of
a degree increasing sequence $J$,
is invariant with respect to all mKdV flows.

In Section \ref{sec SchuR poly} we discuss properties of Schur polynomials.
Schur polynomials are labeled by partitions. Following \cite{SW} we consider partitions
$\la=(\la_0\geq \la_1\geq\dots)$ and associated subsets $S=\{s_0<s_1<\dots\} \subset\Z$ of virtual cardinal zero.
Since partitions are in a one-to-one correspondence with subsets of virtual cardinal zero, we label Schur polynomials
by subsets too.
Theorem \ref{thm new identity} gives a Wronskian identity for Schur polynomials associated with four suitable
subsets $S_1, S_2, S_3, S_4$ of virtual cardinal zero:\
\bea
\Wr_{t_1}(F_{S_1},F_{S_2}) = F_{S_3}F_{S_4},
\eea
 see Section \ref{Subsets of virtual cardinal zero}. More general Wronskian identities for Schur polynomials see in
 Section \ref{sec more general}.

We fix a natural number $N>1$ and introduce KdV subsets $S\subset \Z$ as subsets of virtual cardinal zero with the property $S+N\subset S$.
An example is $S^\emptyset=\{0<1<\dots\}$.
We introduce mKdV tuples $\bs S=(S_1,\dots,S_N)$ as tuples of KdV subsets with the property  $S_i+1\subset S_{i+1}$ for all $i=1,\dots,N$.
An example is $\bs S^\emptyset = (S^\emptyset,\dots, S^\emptyset)$. We denote $\mc S_{mKdV}$ the set of all mKdV tuples of subsets.
For any  $\bs S=(S_1,\dots,S_N)\in \mc S_{mKdV}$  and $i=1,\dots,N$ we define a mutation $w_i: \bs S\mapsto \bs S^{(i)}=(S_1,\dots,\tilde S_i,\dots,S_N)$.
The mutations $w_1,\dots,w_N$ satisfy the relations of the affine Weyl group $\widehat W_{A_{N-1}}$ and define a transitive action of
 $\widehat W_{A_{N-1}}$ on $\mc S_{mKdV}$, see Theorems \ref{thm mutation all} and \ref{lem weyl action}.

Theorem \ref{thm schur-crit} gives a relation of mKdV tuples of subsets with critical points. Namely, if
$\bs S=(S_1,\dots,S_N)\in \mc S_{mKdV}$  and $(F_{S_1}(t_1,t),\dots,F_{S_N}(t_1,t))$ is the tuple
of the corresponding Schur polynomials, then for a fixed generic $t$, the tuple $(F_{S_1}(x,t),$ $\dots,F_{S_N}(x,t))$
represents a critical point of a master function. Theorem \ref {thm schur induced} says that every family of critical
points provided by Theorem \ref{thm schur-crit}
appears as a subfamily of critical points generated from the tuple $y^\emptyset$.

Let $\bs S=(S_1,\dots,S_N)\in \mc S_{mKdV}$. Consider the differential operator
\bean
\label{THREE}
&&  \D_S = \left(\frac d{dx} - \log'\left(\frac{F_{S_N}(x,t)}{F_{S_{N-1}}(x,t)}\right)\right) \times
\\
\notag
&&
\phantom{
ghjjjjjjjjjjj\dots}
\times\left(\frac d{dx} - \log'\left(\frac{F_{S_{N-1}}(x,t)}{F_{S_{N-2}}(x,t)}\right)\right)
\dots
 \left(\frac d{dx} - \log'\left(\frac{F_{S_{1}}(x,t)}{F_{S_{N}}(x,t)}\right)\right)
\eean
with respect to $x$. Theorem \ref{thm diff oper SCur} says that this differential operator depends  on $S_N$ only.
Theorem \ref{thm diff oper SCur} also describes the kernel of this differential operator as the span of ratios of suitable Schur polynomials.

In Section \ref{Critical points and tau functions} following \cite{SW} we consider the Hilbert  space
 $H=L^2(S^1)$, the subspace $H_+$, which is the closure of the span of $\{z^j\}_{j\geq 0}$,
 and  the set $\GR$ of all  closed subspaces $W \subset H$ such that
$z^q H_+ \subset W \subset z^{-q} H_+$  for some $q>0$. We define $\Gr\subset \GR$ as the subset of subspaces of virtual dimension zero.
 Following \cite{SW} we define the tau-functions
$\tau_W(t_1,t)$ of subspaces $W \in \Gr$. The tau-functions are polynomials in $t_1,t$.
Theorem \ref{finiteness} says that the tau-function $\tau_W(t_1,t=0)$ determines $W$ up to a finite number of possibilities.

Theorem \ref{thm new identity tau} gives a Wronskian identity for tau-functions associated with four suitable
subspaces $W_1, W_2, W_3, W_4\in\Gr$:\
\bea
\Wr_{t_1}(\tau_{W_1}, \tau_{W_2})= \on{const}\, \tau_{W_3}\tau_{W_4},
\eea
 see Section \ref{proP}.

 We fix a natural number $N>1$ and
 define
KdV subspaces as subspaces $W\in\Gr$ such that $z^NW\subset W$. An example is $H_+$.
We define mKdV tuples $\bs W=(W_1,\dots,W_N)$ of subspace as tuples of KdV subspaces with the property $zW_i\subset W_{i+1}$ for
all $i$. An example is $\bs W^\emptyset=(H_+,\dots,H_+)$.

In Section \ref{sec mkdv tuples of subspaces} we formulate an important theorem by G.\,Wilson, see \cite{W1}. Let
 $\bs W=(W_1,\dots,W_N)\in\GM$. Let  $(\tau_{W_1},\dots,\tau_{W_N})$ be the tuple of the
corresponding tau-functions.
Consider the Miura oper $\Ll_{\bs W} = \frac d{dx} + \Lambda + V$ where
\bea
V = \on{diag} \left( \log'\left(\frac{ \tau_{W_1}(x+t_1,t)}{ \tau_{W_N}(x+t_1,t)}\right),
\log'\left(\frac{ \tau_{W_2}(x+t_1,t)}{ \tau_{W_1}(x+t_1,t)}\right),\dots,
\log'\left(\frac{ \tau_{W_N}(x+t_1,t)}{ \tau_{W_{N-1}}(x+t_1,t)}\right)
\right).
\eea
Then $\Ll_W$  satisfies all mKdV equations, see Theorem \ref{thm Wilson}.

In Sections \ref{sec Gentions of  mKdV subspaces } and \ref{sec norm gen} we define the generation of new mKdV tuples of subspaces starting
with a given mKdV tuple.
Theorem \ref{thm exist}
 says that the generation acts transitively on the set of mKdV tuples of subspaces.
Theorem \ref{thm induction} describes a relation of the generation of new mKdV tuples of subspaces and
the generation of new critical points  of master functions. Theorem \ref{thm induction} together
with G. Wilson's Theorem \ref{thm Wilson} give a new proof of the main result of the paper,  Corollary \ref{cor Main}.

Corollary \ref{cor last} says that points of the set of all mKdV tuples are in one-to-one correspondence with the
tuples $(y_1,\dots,y_N)$ of the population of tuples of polynomials in $x$ generated from $y^\emptyset=(1,\dots,1)$, the definition of the population
see in Section \ref{sec generation procedure}.
 The correspondence is
$(W_1,\dots,W_n)\mapsto (\tilde\tau_{W_1}(t_1=x,t=0),\dots,\tilde \tau_{W_N}(t_1=x,t=0))$,
where $\tilde \tau_W(t_1,t)$ is the normalized tau-function of $W$ defined in Section \ref{proP}.

\section{KdV and mKdV hierarchies}
\label{sec kdv and mkdv}

\subsection{Algebra $\text{Mat}$}
\label{sec algebra}

Denote $\Bb$ the space of complex-valued functions of one variable $x$.  Given a finite dimensional vector space $U$, denote $\Bb(U)$ the space of
$U$-valued functions of $x$.
Let $\C((\lambda^{-1}))$ be the algebra of formal Laurent series
$\sum_{i \in \Z} c_i \lambda^{i}, c_i \in \C$, where each sum has finitely many terms with $i>0$.

Denote $\text{Mat}$ the algebra of $N \times N$ matrices $\sum_{i,j = 1}^N c_{i,j} e_{i,j}$
with basis $e_{ij}$, coefficients $c_{i,j}$ in $\C((\lambda^{-1}))$
and the product $e_{i,j} e_{k,l} = \delta_{j,k} e_{i,l}$.
Denote $\text{Diag}\subset \text{Mat}$ the subalgebra of diagonal matrices
 $\sum_{i=1}^N c_i e_{i,i}$
 with $c_i \in \C$. Denote $\text{Diag}_0\subset \text{Diag}$ the vector subspace of matrices with zero trace.
Denote
$
\Lambda = e_{2,1} + \dots + e_{N,N-1} + \lambda e_{1,N}
\ \in  \text{Mat}$,
see \Ref{LAM}.

\begin{lem}[\cite{DS}] \label{decomplemma}
Each element $M\in \text{Mat}$ can be uniquely presented in the form $\sum_{i \in \Z} d_i \Lambda^{i}$
with only finitely many terms $i > 0$ and all $d_i \in \text{Diag}$.
\end{lem}

Given
$M = \sum_{i \in \Z} d_i \Lambda^{i}$
with $d_i\in \di$, define $M^+ = \sum_{i \geq 0} d_i \Lambda^{i}$, $M^- = \sum_{i < 0} d_i \Lambda^{i}$, and
\bean
\label{0}
M^0 = d_0.
\eean

Define the following elements of $\dio$:\ {} $H_i=e_{i,i}-e_{i+1,i+1}$,
$i=1,\dots,N-1$, and $H_N=e_{N,N}-e_{1,1}$.
Define elements $\al_i\in \dio^*$ for $i=1,\dots,N$ by
\bea
&&
\langle \alpha_i, H_i \rangle = 2, \qquad
 \langle \alpha_i, H_{j} \rangle = -1
 \quad\text{if}
 \quad i - j= \pm 1 \ \text{mod}\, N,
 \\
&&
\phantom{aaaaa}
 \langle \alpha_i, H_j \rangle = 0\quad\text{if}
 \quad i\ne j\ \text{and}\ i-j \neq \pm 1 \ \text{mod}\, N.
\eea

 Introduce the following elements of $\text{Mat}$:
  \ $E_i=e_{i,i+1}$, $F_i=e_{i+1,1}$ for $i=1,\dots,N-1$ and $E_N=\lambda^{-1} e_{N,1}$,
   $F_N=\lambda e_{1,N}$. For $j=1,\dots,N$, denote  $\Nn^+_j$ the subalgebra of
   $\text{Mat}$ generated by the
elements $E_i$ with  $i \in \{1,\dots,N\}$ and $i\neq j$.

Throughout this paper, we consider all indices modulo $N$ when appropriate.
Thus we will write $H_0, E_0$ or $H_N, E_N$ interchangeably.

\begin{lem}
\label{lem exp}
 Let $g \in \Bb$ and $j \in \{1,\dots,N\}$. Then
 \bean
 \label{formula exp}
 e^{gE_j} = 1 + ge_{j,j}\Lambda^{-1}.
 \eean
 \qed
\end{lem}

\begin{lem}
\label{lem lambda} We have
$\La^{-1} = e_{1,2}+\dots+ e_{N-1,N} + \la^{-1}e_{N,1}$, and
\bean
\label{formula La}
e_{i+1,i+1}\La = \La e_{i,i}, \qquad
e_{i,i}\La^{-1} = \La^{-1}e_{i+1,i+1}
\eean
for all $i$.
 \qed
\end{lem}

\subsection{Generalized mKdV Hierarchy, \cite{DS}}
Denote $\der$ the differential operator $\frac{d}{dx}$.
For $j \in\{ 1,\dots,N\}$, a \em $j$-oper \em is a differential operator of the form
\bea
\Ll = \der + \Lambda + V + W
\eea
with $V \in \Bb(\dio)$ and $W \in \Bb(\Nn^+_j)$.
For $w \in \Bb(\Nn^+_j)$ and a $j$-oper $\Ll$, the differential operator
\bea
e^{\text{ad} \, w} \Ll = \Ll + [w,\Ll] + \frac{1}{2}[w,[w,\Ll]] + \dots
\eea
is a $j$-oper.  The $j$-opers $\Ll$ and $e^{\text{ad} \, w} \Ll$ are called \em $j$-gauge equivalent. \em
As $\Nn^+_j$ is nilpotent, $e^w$ and $e^{-w}$ are well defined and
\bea
e^{\text{ad}\,  w} \Ll = e^w \Ll e^{-w}.
\eea
 A \em Miura oper \em is a differential operator of the form
\bea
\Ll = \der + \Lambda + V
\eea
with $V \in \Bb(\dio)$.  A Miura oper is a $j$-oper for any $j$.
 Denote by $\mc{M}$ the space of all Miura opers.

\begin{lem}[\cite{DS}]
\label{Tlemma}
Let $\Ll$ be a Miura oper.

\begin{enumerate}
\item[(i)]
Then there exists $T = 1 + \sum_{i < 0} T_i \Lambda^i$ with $T_i \in \Bb(\di)$ such that
 $T^{-1} \Ll T$ has the form $\der + \Lambda + \sum_{i \leq 0} b_i \Lambda^i$ with $b_i\in \Bb$.

\item[(ii)]
The matrix $T$ is not unique.  If both $T$ and $\widetilde{T}$ have property (i), then \\
$\widetilde{T} =
 (1 + \sum_{i < 0} q_i \Lambda^i)T$ with $q_i \in \Bb$.  In particular, for any $r \in \N$ the matrix $T\Lambda^rT^{-1}$ does not depend on
the choice of $T$.
\end{enumerate}
\end{lem}

Let $\Ll$ be a Miura oper and $r\in\N$. The differential equation
\bean
\label{mKdVr}
\frac{\partial \Ll}{\partial t_r}  = [\Ll,\left(T \Lambda^r T^{-1} \right)^+]
\eean
 is called \em the $r$-th  mKdV equation\em, see \cite{DS}.
  The classical mKdV equation corresponds to the case $N=2$, $r=3$.

Equation \Ref{mKdVr}
defines vector fields $\frac{\partial}{\partial t_r}$ on the space $\mc M$ of Miura opers.
 For all $r,s \in \N$ the vector fields $\frac{\partial}{\partial t_r}$ and $\frac{\partial}{\partial t_s}$ commute, see \cite{DS}.

\begin{lem}[\cite{DS}]
\label{lem der}
We have
\bean
\label{mKdVr}
\frac{\partial\Ll}{\partial t_r}  = \frac d{dx}(T \Lambda^r T^{-1})^0,
\eean
where ${}^0$ is defined in \Ref{0}.

\end{lem}

\subsection{Generalized KdV Hierarchy, \cite{DS}}
Let $\Bb((\der^{-1}))$ be the algebra of formal pseudodifferential operators of the form
 $a=\sum_{i \in \Z} a_i \der^i$, with $a_i \in \Bb$ and finitely many terms with $i > 0$.
 The relations in this algebra are
\bea
\partial^k u - u \partial^k = \sum_{i = 1}^\infty k(k-1)\dots(k-i+1)\frac{d^i u}{dx^i}\partial^{k-i}
\eea
for any $k\in\Z$ and $u\in\Bb$.
For $a = \sum_{i \in \Z} a_i \der^i \in \Bb((\der^{-1}))$, define $a^+ = \sum_{i \geq 0} a_i \der^i$.

Denote $\Bb[\der] \subset \Bb((\der^{-1}))$ the subalgebra of differential operators $a = \sum_{i=0}^m a_i \der^i$
with $m\in\Z_{\geq 0}$. Denote $\mc D \subset \Bb[\der]$ the affine subspace of the differential operators of the
form $L=\der^N + \sum_{i = 0}^{n-2} u_i \der^i$.

For $L \in \D$, there exists a unique $L^{\frac{1}{N}} =\der + \sum_{i\leq 0}a_i\der^i \in \Bb((\der^{-1}))$
 such that  $(L^{\frac{1}{N}})^N = L$.
For $r\in \N$, the differential equation
\beq
\label{KdVr}
\frac{\partial L}{\partial t_r} = [(L^{\frac{r}{N}})^+,L]
\eeq
is called  \em the $r$-th  KdV equation \em.
The classical KdV equation corresponds to the case $N=2$, $r=3$.

 Equation \Ref{KdVr} defines flows $\frac{\partial}{\partial t_r}$ on the space $\D$.
  For all $r,s \in \N$ the flows $\frac{\partial}{\partial t_r}$ and $\frac{\partial}{\partial t_s}$ commute, see \cite{DS}.

\subsection{Miura maps}
\label{sec Miura maps}

Let $\Ll = \der + \Lambda + V$ be a Miura oper with $V = \sum_{k = 1}^N v_k e_{k,k}$.  For $i=1,\dots,N$, define
a scalar differential operator $L_i\in\D$ by the formula
\bean \label{miuramap}
L_i &=& (\der - v_i)(\der - v_{i-1})\dots(\der - v_1)(\der - v_N)\dots(\der - v_{i - 2})(\der - v_{i-1})
\\
 &=& \der^N + \sum_{k = 0}^{N-2}u_{k,i}\der^k.
\notag
\eean
Recall that  we have $\sum_{k=1}^Nv_k=0$
by the definition of a Miura oper.

\begin{thm}[\cite{DS}] \label{thm mkdvtokdv}
Let a Miura oper $\Ll$ satisfy the mKdV equation \Ref{mKdVr} for some r.  Then for every $i=1,\dots,N$ the differential operator
$L_i$ satisfies the KdV equation \Ref{KdVr}.
\end{thm}

We define the {\it  $i$-th Miura map} by the formula
\bea
\frak{ m}_i \ : \ \mc M\  \to \ \D,
\quad
 \Ll \ \mapsto \ L_i,
 \eea
  see \Ref{miuramap}.

\begin{thm}
[\cite{DS}]
\label{thm gaugemiura}
If Miura opers $\Ll$ and $\tilde \Ll$ are $i$-gauge equivalent, then
$\frak m_i(\Ll) = \frak m_i(\tilde\Ll)$.
\end{thm}

\section{Critical Points of Master Functions, \cite{MV1}}
\label{CP}
\subsection{Master function}
Choose a collection of nonnegative integers $\bs{k} = (k_1,\dots,k_{N}).$
Consider variables $ u = (u_i^{(j)})$, where $j = 1,\dots,N,$ and $i = 1,\dots,k_j$.
We adopt the notations $k_{N+1} = k_{1}$ and  $u_i^{(N+1)} = u_i^{(1)}$ for all $i$.
 The {\it master function} $\Phi(u; \bs k)$
is defined by formula \Ref{Master}:
\bea
\Phi(u,\bs k) =
2 \sum_{j=1}^N  \sum_{i<i'}
\log (u^{(j)}_i-u^{(j)}_{i'})
- \sum_{j=1}^{N} \sum_{i,i'}\log (u^{(j)}_i-u^{(j+1)}_{i'}) .
\eea
The product of symmetric groups
$\Sigma_{\bs k}=\Si_{k_1}\times \dots \times \Si_{k_N}$ acts on the set of variables
by permuting the coordinates with the same upper index.
The  function $\Phi$ is symmetric with
respect to the $\Si_{\bs k}$-action.

A point $u$ is a {\it critical point} if $d\Phi=0$ at $u$. In other words, $u$ is a critical point if and only if
\bean
\label{Bethe eqn 1}
\sum_{i' \neq i}\frac {2}{ u_i^{(j)} - u_{i'}^{(j)}}
-\sum_{i=1}^{k_{j+1}} \frac{1}{ u_i^{(j)} -u_{i'}^{(j+1)}} -\sum_{i=1}^{k_{j-1}} \frac{1}{ u_i^{(j)} -u_{i'}^{(j-1)}}
= 0,
\eean
for $j = 1, \dots , N$ and $i = 1, \dots , k_j$. The critical set is $\Si_{\bs k}$-invariant.
All orbits have the same cardinality $k_1! \cdots k_N!$\ .
We do not make distinction between critical points in the same orbit.

\subsection{Polynomials representing critical points}
\label{PRCP}

Let $u = (u_i^{(j)})$ be a critical point of the master function
$\Phi$.
Introduce an $N$-tuple of polynomials $ y=( y_1(x),$ $\dots ,$ $ y_{N}(x))$,
\bean\label{y}
y_j(x)\ =\ \prod_{i=1}^{k_j}(x-u_i^{(j)}).
\eean
We adopt the notations $y_{N+1} = y_{1}$ and $y_0 = y_N$.  Each polynomial is considered up to multiplication
by a nonzero number.
The tuple defines a point in the direct product
$\PCN$.
We say that the tuple {\it represents the
critical point}.

It is convenient to think that the tuple $\bs{y}^\emptyset = (1, \dots , 1)$ of constant polynomials
represents  in $\PCN$ the critical point of
the master function with no variables.
This corresponds to the case  $\bs k = (0, \dots , 0)$.

We say that a given tuple $ y$ is {\it generic} if each polynomial $y_j(x)$ has no multiple roots and
the polynomials  $y_j(x)$ and
 $y_{j+1}(x)$ have no common roots for $j=1,\dots, N$. If a tuple represents a critical point, then it is generic,
 see \Ref{Bethe eqn 1}.  For example, the tuple  ${y}^\emptyset $ is generic.

\subsection{Elementary generation}
\label{Elementary generation}

A tuple $ y$ is called {\it fertile} if for every $j = 1,\dots,N$ there exists a polynomial $\tilde y_j$ such that
\bean
\label{wronskian-critical eqn}
\Wr(y_j, \tilde y_j)\ =\ y_{j - 1} y_{j + 1}.
\eean
For example, the tuple $y^\emptyset$ is fertile.

Equation \Ref{wronskian-critical eqn} is a first order inhomogeneous differential equation with respect to $\tilde y_j$.
Its solutions are
\bean
\label{deG}
\tilde y_j = y_j\int \frac{y_{j-1}y_{j+1}}{y_j^2} dx + c y_j,
\eean
where $c$ is any number.
The tuples
\bean
\label{simple}
y^{(j)}(x,c) = (y_1(x), \dots , \tilde y_j(x,c),\dots, y_{N}(x))
\quad \in  \quad \PCN \
\eean
 form a one-parameter family.  This family  is called
 the {\it generation  of tuples  from $ y$ in the $j$-th direction}.
 A tuple of this family is called an  immediate descendant of $ y$ in the $j$-th direction.

\begin{thm}
[\cite{MV1}]
\label{fertile cor}
${}$

\begin{enumerate}
\item[(i)]
A generic tuple $y = (y_1, \dots , y_{N})$
represents a critical point of the master function
if and only if $y$ is fertile.

\item[(ii)] If $ y$ represents a critical point,
$j \in \{1,\dots,N\}$, and $ y^{(j)}$ is
an immediate descendant of $y$, then $y^{(j)}$ is fertile.

\item[(iii)]
If $y$ is generic and fertile, then for almost all values of the parameter
 $c\in \C$ the corresponding $n$-tuple ${y}^{(j)}$ is generic.
The exceptions form a finite set in $\C^1$.

\item[(iv)]

Assume that a sequence ${y}_i, i = 1,2,\dots$ of fertile $N$-tuples
has a limit ${y}_\infty$ in $\PCN$ as $i$ tends to infinity.
\begin{enumerate}
\item[(a)]
Then the limiting $N$-tuple ${y}_\infty$ is fertile.
\item[(b)]
For $j = 1,\dots,N$, let ${y}_\infty^{(j)}$ be an immediate
descendant of ${y}_\infty$.
Then for all $j$ there exist immediate descendants
${y}_i^{(j)}$ of ${y}_i$ such that ${y}_\infty^{(j)}$
is the limit of ${y}_i^{(j)}$ as $i$ tends to infinity.

\end{enumerate}

\end{enumerate}
\end{thm}

\subsection{Degree increasing generation}
\label{Degree increasing generation}

For $j=1,\dots,N$, define $k_i=\deg y_j$. The polynomial $\tilde y_j$ in \Ref{deG}
is of degree $k_j$ or
$\tilde k_j=k_{j-1} + k_{j+1}+ 1 - k_j$. We say that the {\it generation is
  degree increasing }  if $\tilde k_j > k_j$. In that
case $\deg \tilde y_j=\tilde k_j$ for all $c$.

If the generation is degree increasing we will normalize family
 \Ref{simple} and construct a map
$ Y_{y,j} : \C \to (\C[x])^N$ as follows. First we multiply the polynomials $y_1,\dots,y_N$ by numbers to make them monic.
 Then choose a monic polynomial $ y_{j,0}$ satisfying the equation $\Wr(y_j,  y_{j,0})$
 $=\,\on{const}\,y_{j-1}y_{j+1}$
 and such that the coefficient of $x^{k_j}$ in $\tilde y_{j,0}$ equals zero. Set
 \bean
 \label{tilde y}
 \tilde y_j(x,c)=y_{j,0}(x) + cy_j(x)
 \eean
 and define
\bean
\label{normalized}
 Y_{y,j} \ :\ \C\ \to\ (\C[x])^N, \qquad c \mapsto\ y^{(j)}(x,c) = (y_1(x),\dots,\tilde y_j(x,c),\dots, y_N(x)).
 \eean
All polynomials of the tuple $y^{(j)}(x,c)$ are monic.

\begin{lem}
\label{lem exist increa}
If  $k_j$ is a minimal number among $k_1,\dots,k_N$, then
the generation of tuples from $y$ in the $j$-th direction is degree
increasing.
\qed
\end{lem}

\subsection{Degree-transformations and generation of vectors of integers}
\label{sec degree transf}

For $j=1,\dots,N$, the degree-transformation
\bean
\label{l-transformation}
\phantom{aaaaaa}
\bs k=(k_1,\dots,k_N)\ \mapsto
\bs k^{(j)}=(k_1,\dots,k_{j-1},k_{j-1}+k_{j+1}-k_j+1,k_{j+1},\dots, k_N)
\eean
corresponds to the shifted action of the affine reflection $ w_j\in \widehat W_{A_{N-1}}$,
where $\widehat W_{A_{N-1}}$ is the affine Weyl group of type $A_{N-1}$ and $w_1,\dots, w_N$
are its standard generators, see Lemma 3.11 in \cite{MV1} for more detail.

We take formula \Ref{l-transformation} as the definition of {\it degree-transformations}:
\bean
\label{s_i l-transf}
 w_j\ :\ \bs k=(k_1,\dots,k_N)\ \mapsto
\bs k^{(j)}=(k_1,\dots,k_{j-1}+k_{j+1}-k_j+1,\dots, k_N)
\eean
for $j=1,\dots,N,$ acting on arbitrary vectors $\bs k=(k_1,\dots,k_N)$.
Here we consider the indices of
the coordinates modulo $N$, thus we have $k_0=k_N, k_{N+1}=k_1$ and so on.

\medskip

We start with the vector $\bs k^\emptyset=(0,\dots,0)$ and a sequence $J=(j_1,j_2,\dots,j_n)$ of integers,
$1\leq j_i\leq N$. We apply the corresponding degree transformations to  $\bs k^\emptyset$ and obtain
a sequence of vectors $\bs k^\emptyset,$\ $  \bs k^{(j_1)} =w_{j_1}\bs k^\emptyset, $\ $
  \bs k^{(j_1,j_2)} = w_{j_2} w_{j_1}\bs k^\emptyset$,\dots
\bean
\label{gen vector}
\bs k^J  = w_{j_m}\dots w_{j_2} w_{j_1}\bs k^\emptyset .
\eean
We say that the {\it vector $\bs k^J $ is generated from $(0,\dots,0)$ in the direction of $J$}.

\medskip
For example, for $N=3$ and $J=(1,2,3,1,2,3)$ we get the sequence
$(0,0,0),$ $ (1,0,0),$ $ (1,2,0),$ $ (1,2,4), (6,2,4), (6,9,4),\ \bs k^J = (6,9,12)$.
\medskip
We call the sequence $J$ {\it degree increasing} if for every $i$ the transformation
$w_{j_i}$ applied to  $  w_{j_{i-1}}\dots w_{j_1}\bs k^\emptyset$
increases the $j_i$-th coordinate.

The sequence $J=(j_1,\dots,j_m)$ with  $1\leq j_i \leq N$ will be called {\it cyclic} if
$j_1=1$ and $j_{i+1} = j_i+1$ mod $N$ for all $i$.

\begin{lem}
\label{lem cyclic increas}
A cyclic sequence $J=(j_1,\dots,j_m)$ is degree increasing.
If $m = p(N-1) + q$ for some integers $p,q$ such that
$0 \leq q \leq N - 2$, then the last changed coordinate of $\bs k^J$
has index equal to $m$ mod $N$ and this coordinate equals
\bea
 (N-1)\frac{p(p+1)}{2} + (p+1)q.
\eea
\qed
\end{lem}

For example, for $N=3$,  $J=(1,2,3,1,2,3)$, $\bs k^J = (6,9,12)$, we have $p=3, q=0$,
and the last changed third coordinate equals 12.

\subsection{Multistep generation}
\label{sec generation procedure}

Let $J = (j_1,\dots,j_m)$ be a degree increasing sequence of integers.
Starting from $y^\emptyset=(1,\dots,1)$ and $J$, we will construct by induction on $m$
a map
\bea
Y^J : \C^m \to (\C[x])^N.
\eea
If $J=\emptyset$, the map $Y^\emptyset$ is the map $\C^0=(pt)\ \mapsto y^\emptyset$.
If $m=1$ and $J=(j_1)$,  the map
$Y^{(j_1)} :  \C \to (\C[x])^N$ is given by formula \Ref{normalized}
for $y=y^\emptyset$ and $j=j_1$. More precisely, equation \Ref{wronskian-critical eqn}
takes the form $\Wr(1, \tilde y_{j_1}) =1$. Then $\tilde y_{j_1,0}= x$ and
\bea
Y^{(j_1)}\ :\ \C \mapsto (\C[x])^N, \qquad
c \mapsto (1,\dots,1,x+c, 1\dots,1),
\eea
where $x+c$ stands at the $j_1$-th position. By Theorem \ref{fertile cor}
all tuples in the image are fertile and almost all tuples are generic
(in this example all tuples are generic).

Assume that for ${\tilde J} = (j_1,\dots,j_{m-1})$,  the map
$Y^{{\tilde J}}$ is constructed. To obtain  $Y^J$ we apply the
generation procedure in the $j_m$-th
direction to every tuple of the image of $Y^{{\tilde J}}$. More precisely, if
\bean
\label{J'}
Y^{{\tilde J}}\ : \
{\tilde c}=(c_1,\dots,c_{m-1}) \ \mapsto \ (y_1(x,{\tilde c}),\dots, y_N(x,{\tilde c})).
\eean
Then
\bean
\label{Ja}
&&
{}
\\
\notag
&&
Y^{J} : \C^m \mapsto (\C[x])^N, \quad
({\tilde c},c_m) \mapsto
(y_1(x,{\tilde c}),\dots,  y_{j_m,0}(x,{\tilde c}) + c_m y_{j_m}(x,{\tilde c}),\dots,
y_N(x,{\tilde c})),
\eean
see formula \Ref{tilde y}.
The map  $Y^J$ is called  the {\it generation  of $N$-tuples   from $ y^\emptyset$ in the $J$-th direction}.

\begin{lem}
\label{lem gen procedure}
All tuples in the image of $Y^J$ are fertile and almost all tuples are generic. For any $c\in\C^m$
the $N$-tuple $Y^J(c)$ consists of monic polynomials. The degree vector of this tuple
equals $\bs k^J$, see \Ref{gen vector}.
\qed
\end{lem}

\begin{lem}
\label{lem uniqeness}
The map  $Y^J$ sends distinct points of $\C^m$ to distinct points of  $(\C[x])^N$.
\end{lem}

\begin{proof}
The lemma is easily proved by induction on $m$.
\end{proof}

\begin{example} \label{example AM}
Let $N=2$, $J=(1,2)$. Then
\bea
Y^{(1)}(c_1)= (x+c_1,1), \qquad
Y^{(1,2)}(c_1,c_2)=
(x+c_1, x^3+3c_1x^2+3c_1^2x + c_2),
\eea
cf. with the first two Adler-Moser polynomials in \cite{AM}.
\end{example}

\begin{example}
 Let $N=3$, $J=(1,2,3)$. Then
\bea
Y^{(1)}(c_1)= (x+c_1,1,1), \qquad
Y^{(1,2)}(c_1,c_2)=
(x+c_1, x^2+2c_1x+ c_2, 1),
\eea
\bea
Y^{(1,2,3)}(c_1,c_2,c_3)=
(x+c_1, x^2+2c_1x+ c_2,
x^4+4c_1x^3+(2c_2+4c_1^2)x^2+4c_1c_2x+ c_3),
\eea
cf. example in Section \ref{sec gene for gh}.
\end{example}

The set of all tuples $(y_1,\dots,y_N)\in (\C[x])^N$ obtain from $y^\emptyset=(1,\dots,1)$
by generations in all degree increasing directions will be called the {\it population of tuples}
generated from $y^\emptyset$.

\section{Miura opers and  critical points, \cite{MV2}}
\label{MOCP}

\subsection{Deformations of Miura opers}

\begin{lem}[\cite{MV2}]
Let $\Ll = \der + \Lambda + V$ be a Miura oper.  Let $g \in \Bb$ and $j \in \{1,\dots,N\}$.  Then
\bean \label{adeq}
e^{\text{ad} (g E_j)} \Ll = \der + \Lambda + (V + g H_j) - (g^\prime + \langle \alpha_j , V \rangle g + g^2)E_j.
\eean
\end{lem}

The opers $\Ll$ and $e^{\text{ad}( g E_j )} \Ll$ are $i$-gauge equivalent for every $i\neq j$.

\begin{cor}[\cite{MV2}]
Let $\Ll = \der + \Lambda + V$ be a Miura oper. Then $e^{\text{ad}( g E_j )} \Ll$
is a Miura oper if and only if the scalar function
$g$ satisfies the Ricatti equation
\bean
\label{Ric}
g' + \langle \alpha_j , V \rangle g +  g^2 = 0 \ .
\eean
\end{cor}

Let $\Ll = \der + \Lambda + V$ be a Miura oper with $V=\sum_{i=1}^{N} v_ie_{i,i}$. Assume that the functions $v_i$ are rational functions of $x$.
For $j\in\{1,\dots,N\}$, we say that $\Ll$ is  {\it deformable} in the $j$-th direction
if equation \Ref{Ric} has a nonzero solution $g$, which is a rational function.

\subsection{Miura oper associated with a tuple}
Define a map
\bea
\mu \ : \ (\C[x])^N\ \to\ \mc M ,
\eea
which sends a tuple $y=(y_1,\dots,y_N)$ to the Miura oper $\Ll=\der+\Lambda+V$ with
\bea
V \ =\ \sum_{k=1}^{N} \log'(y_k)\,H_k\  =\  \sum_{k=1}^N \log'(y_k/y_{k-1})\, e_{k,k} .
\eea
Here $\log'(f)$ denotes $f'/f$. The Miura oper $\mu(y)$ determines $ y$ uniquely if $ y$ is generic.
The Miura oper $\mu( y)$ is called  {\it associated} with $y$.
For example,
\bea
\Ll^\emptyset = \der + \Lambda
\eea
is associated with $y^\emptyset = (1,\dots,1)$.

\begin{thm}
[\cite{MV2}]
\label{thm deform}
Let a Miura oper $\Ll$ be associated with $y=(y_1,\dots,y_N)$. Let $j\in\{1,\dots,N\}$. Then
$\Ll$ is deformable in the $j$-th direction if and only if there exists a polynomial $\tilde y_j$ satisfying
equation \Ref{wronskian-critical eqn}. Moreover, in that case any nozero rational solution $g$ of equation \Ref{Ric} has the form
$g=\log'(\tilde y_j/y_j)$, where $\tilde y_j$ is a solution of equation \Ref{wronskian-critical eqn}.
If $g=\log'(\tilde y_j/y_j)$, then the Miura oper
\bea
e^{\text{ad}(gE_j)}\Ll = \Ll + gH_j
\eea
is associated with the tuple $y^{(j)}=(y_1,\dots,\tilde y_j,\dots,y_N)$, see Section \ref{Elementary generation}.

\end{thm}

In particular, if $\Ll$ is associated with a fertile tuple, then $\Ll$ is deformable in all directions.

\begin{lem}
\label{lem v' v}
Let an oper $\der +\La+\sum_{k=1}^N v_k e_{k,k}$ be associated with an $N$-tuple of polynomials $y$, that is assume
that $v_k = \log'(y_k/y_{k-1})$. If $v'_k=0$ for some $k$, then $v_k=0$ and $y_k=a y_{k-1}$ for some $a\in\C$.
\end{lem}

\begin{proof}
If $v_k'=0$, then $v_k=b$ for some $b\in\C$. Then $y_k=ae^{bx}y_{k-1}$ for some $a\in\C$. Hence $b=0$ since $y_k,y_{k-1}$ are
polynomials.
\end{proof}

\subsection{Miura opers associated with the generation procedure}

\label{Miura opers with cr points}

Let $J = (j_1,\dots,j_m)$ be a degree increasing sequence of integers.
Let $Y^J : \C^m \to \PCN $
be the generation of tuples from $y^\emptyset$ in the $J$-th direction.
We define the associated family of Miura opers by the formula:
\bea
\mu^J \  :\ \C^m\ \to\ \mc M, \qquad c \ \mapsto \ \mu(Y^J(c)) .
\eea
The map $\mu^J$ is called the {\it generation of Miura opers from $\Ll^\emptyset$
in the $J$-th direction.}

For $\ell=1,\dots,m$, denote $J_\ell=(j_1,...,j_\ell)$ the beginning $\ell$-interval of
the sequence $J$. Consider the associated map $Y^{J_\ell} : \C^\ell\to\PCN$. Denote
\bea
Y^{J_\ell}(c_1,\dots,c_\ell) = (y_1(x,c_1,\dots,c_\ell, \ell),
\dots,y_N(x,c_1,\dots,c_\ell, \ell)).
\eea
Introduce
\bean
\label{g's}
g_1(x,c_1,\dots,c_m) &=&
\log'(y_{j_1}(x,c_1,1)) ,
\\
\notag
g_\ell(x,c_1,\dots,c_m) &=&
\log'(y_{j_\ell}(x,c_1.\dots,c_\ell,\ell)) - \log'(y_{j_\ell}(x,c_1,\dots,c_{\ell-1},\ell-1)),
\eean
for $\ell=2,\dots,m$.
For $c\in\C^m$, denote
\bean
\label{T}
T^J(x,c) = e^{g_m(x,c) E_{j_m}}\cdots e^{g_1(x,c) E_{j_1}}.
\eean

\begin{lem}
\label{lem formula} For $c\in\C^m$, we have
\bean
\label{operformula}
\mu^J(c)\ = \ T^J(x,c) \,\Ll^\emptyset \,(T^J(x,c))^{-1}
\eean
 and
\bean
\label{oper2}
\mu^J(c)\ =\
  \der + \Lambda + \sum_{\ell=1}^m g_\ell(x,c) H_{j_\ell}.
  \eean
\end{lem}

\begin{proof}
The lemma follows from Theorem \ref{thm deform}.
\end{proof}

\begin{example}

Let $N=2$, $J=(1,2)$. Then $\mu^J(c_1,c_2) =\der +\La + \text{diag}(v_1,v_2)$,
where
\bea
v_1 = -v_2 = \log'(x+c_1) - \log'(x^3+3c_1x^2+3c_1^2x + c_2).
\eea
We have $g_1 = \log'(x+c_1),$\ $ g_2 = \log'(x^3+3c_1x^2+3c_1^2x + c_2)$ and
\bea
\mu^J(c)= e^{g_2(x,c) E_{2}} e^{g_1(x,c) E_{1}}
\,\Ll^\emptyset\,e^{-g_1(x,c) E_{1}} e^{-g_2(x,c) E_{2}}.
\eea
\end{example}

\begin{cor}
\label{cor der}
For any $r\in\N$ and $c\in\C^m$, let $\frac {\der\phantom{a}}{\der t_r}\big|_{\mu^J(c)}$ be the value at $\mu^J(c)$ of the vector field of the
$r$-th mKdV flow on the space $\mc M$, see \Ref{mKdVr}. Then
\bean
\label{T mkdv}
\frac {\der\phantom{a}}{\der t_r}\big|_{\mu^J(c)} = \frac {\der\phantom{a}}{\der x}(T^J(x,c) \Lambda^r (T^J(x,c))^{-1})^0,
\eean
where ${}^0$ is defined in \Ref{0}.

\end{cor}
\begin{proof}
The corollary follows from \Ref{mKdVr} and \Ref{operformula}.
\end{proof}

Denote ${\tilde J}=(j_1,\dots,j_{m-1})$. Consider the associated family
$\mu^{{\tilde J}} : \C^{m-1}\to\mc M$.
Let $\frak m_i : \mc M\to \D,\ \Ll \mapsto L_i,$ be Miura maps
defined in Section \ref{sec Miura maps} for  $i=1,\dots,N$.
Denote ${\tilde c}=(c_1,\dots,c_{m-1})$.

\begin{lem}
\label{lem j j'}
We have $\frak m_i\circ \mu^J({\tilde c},c_m) = \frak m_i\circ \mu^{{\tilde J}}({\tilde c})$ for all $i\neq j_m$
and $({\tilde c},c_m)\in\C^m$.

\end{lem}

\begin{proof}
The lemma follows from formula \Ref{operformula} and Theorem \ref{thm gaugemiura}.
\end{proof}

\begin{lem}
\label{lem der c-m}
\bean
\label{form der c-m}
\frac{\der \mu^J}{\der c_m}({\tilde c},c_m) \ =\ a\
\frac{y_{j_m-1}(x,{\tilde c}, m-1) y_{j_m+1}(x,{\tilde c},m-1)}
{y_{j_m}(x,{\tilde c},c_m,m)^2} H_{j_m}
\eean
for some $a\in\C$.
\end{lem}

\begin{proof}
 We have $ y_{j_m}(x,{\tilde c},c_m,m)=y_{j_m,0}(x,{\tilde c}) + c_m y_{j_m}(x,{\tilde c},m-1),$
where $y_{j_m,0}(x,{\tilde c})$ is such that
 \bea
 \Wr (y_{j_m}(x,{\tilde c},m-1), y_{j_m,0}(x,{\tilde c})) \ =\ a\ y_{j_m-1}(x,{\tilde c},m-1)y_{j_m+1}(x,{\tilde c},m-1),
 \eea
 for some $a\in\C$, see \Ref{tilde y}.
We have $ g_m = \log'(y_{j_m}(x,{\tilde c},c_m,m))- \log' (y_{j_m}(x,{\tilde c},m-1))$.

By formula \Ref{oper2}, we have
\bea
\frac{\der \mu^J}{\der c_m}({\tilde c},c_m) = \frac{\der g_m}{\der c_m}({\tilde c},c_m) H_{j_m}
\eea
and
\bea
\frac{\der g_m}{\der c_m}({\tilde c},c_m) &=& \frac \der{\der c_m}\left(
\frac{y_{j_m,0}'(x,{\tilde c}) + c_m y_{j_m}'(x,{\tilde c},m-1)}{y_{j_m,0}(x,{\tilde c}) + c_m y_{j_m}(x,{\tilde c},m-1)}
\right)
\\
&=&
\frac{\Wr ( y_{j_m,0}(x,{\tilde c}), y_{j_m}(x,{\tilde c},m-1))}{(y_{j_m,0}(x,{\tilde c}) + c_m y_{j_m}(x,{\tilde c},m-1))^2}.
\eea
\end{proof}

\section{Vector fields}
\label{Vector fields}

\subsection{Statement}
\label{sec statement}

Recall that  we denote by $ \frac \der{\der t_r}$ the $r$-th mKdV vector field on $\mc M$ and
also  the $r$-th KdV vector field on $\mc D$.
For a Miura map $\frak m_i : \mc M\to \D,\ \Ll \mapsto L_i,$ denote by $d\frak m_i$  the associated
derivative map $T\mc M \to T\D$.  By Theorem
\ref{thm mkdvtokdv} we have $d\frak m_i:  \frac \der{\der t_r}\big|_{\Ll}
\mapsto \frac \der{\der t_r}\big|_{L_i}$.

Fix a degree increasing sequence $J=(j_1,\dots, j_m)$. Consider the associated family $\mu^J:\C^m\to \mc M$
of Miura opers.
For a vector field $\Ga $ on $\C^m$, denote by
$\frac {\der \mu^J}{\der \Ga}$ the derivative of $\mu^J$ along the vector field.
The derivative is well-defined since $\mc M$ is an affine space.

\begin{thm}
\label{thm main}
For every $r\in\N$ there exists a polynomial vector field $\Ga_r$ on $\C^m$
such that
\bean
\label{formula main}
\frac \der{\der t_r}\big|_{\mu^J(c)} = \frac {\der \mu^J}{\der \Ga_r}(c)
\eean
for all $c\in\C^m$. If $r> 2m$, then $\frac \der{\der t_r}\big|_{\mu^J(c)}=0$
for all $c\in\C^m$ and $\Ga_r=0$.

\end{thm}

\begin{cor}
\label{cor Main}
The family $\mu^J$ of Miura opers is invariant with respect to all mKdV flows and
is point-wise fixed by flows with $r>2m$.

\end{cor}

\subsection{Proof of Theorem \ref{thm main} for $m=1$}

Let $J=(j_1)$. Then
\bea
\mu^J(c_1) = e^{g_1E_{j_1}}\Ll^\emptyset e^{-g_1E_{j_1}}
= (1 + g_1e_{j_1,j_1}\La^{-1})\Ll^\emptyset (1 - g_1e_{j_1,j_1}\La^{-1})
= \der + \La + g_1 H_{j_1},
\eea
where $g_1 = \frac 1{x+c_1}$. By formula \Ref{T mkdv},
\bea
\frac \der{\der t_r}\big|_{\mu^J(c_1)}
=\frac {\der}{\der x}((1 + g_1e_{j_1,j_1}\La^{-1}) \Lambda^r (1 - g_1e_{j_1,j_1}\La^{-1}))^0.
\eea
It follows from Lemma \ref{lem lambda} that $\frac \der{\der t_r}\big|_{\mu^J(c_1)}=0$ for $r>1$ and hence
 $\Ga_r=0$. For $r=1$ we have $\frac \der{\der t_1}\big|_{\mu^J(c_1)} = -\frac 1{(x+c_1)^2}H_{j_1}$ and
 this equals $\frac\der{\der c_1}\mu^J(c_1)$. Hence  $\Ga_1= \frac \der{\der c_1}$.
Theorem \ref{thm main} for $m=1$ is proved.

\subsection{Proof of Theorem \ref{thm main} for $m>1$}
\begin{lem}
\label{lem r>2m}
If $r>2m$, then
$\frac \der{\der t_r}\big|_{\mu^J(c)}=0$ for all $c\in\C^m$.

\end{lem}

\begin{proof}
 The vector
 $\frac \der{\der t_r}\big|_{\mu^J(c)}$ equals the right hand side of formula  \Ref{T mkdv}.
By Lemmas \ref{lem exp} and \ref{lem lambda} the right hand side of \Ref{T mkdv} is zero if $r>2m$.
\end{proof}

We prove the first statement of Theorem \ref{thm main} by induction on $m$.
Assume that the statement is proved for ${\tilde J}=(j_1,\dots,j_{m-1})$.
Let
\bea
\label{J'2}
Y^{{\tilde J}}\ : \
{\tilde c}=(c_1,\dots,c_{m-1}) \ \mapsto \ (y_1(x,{\tilde c}),\dots, y_N(x,{\tilde c}))
\eea
be the generation of tuples in the ${\tilde J}$-th direction. Then the generation
of tuples in the $J$-th direction is
\bea
\label{J2}
Y^{J}\ :\ \C^m \mapsto (\C[x])^N, \quad
({\tilde c},c_m) \mapsto
\ (y_1(x,{\tilde c}),\dots,  y_{j_m,0}(x,{\tilde c}) + c_m y_{j_m}(x,{\tilde c}),\dots,
y_N(x,{\tilde c})),
\eea
see \Ref{J'} and \Ref{J}.
We have $g_m = \log'(y_{j_m,0}(x,{\tilde c}) + c_m y_{j_m}(x,{\tilde c}))- \log' (y_{j_m}(x,{\tilde c}))$,
see \Ref{g's}.

By the induction assumption,
there exists a polynomial vector field $\Gamma_{r,{\tilde J}}$ on $\C^{m-1}$ such that
\bean
\label{fromula main m-1}
\frac \der{\der t_r}\big|_{\mu^{{\tilde J}}({\tilde c})} = \frac {\der \mu^{{\tilde J}} }{\der \Ga_{r,{\tilde J}}}({\tilde c})
\eean
for all ${\tilde c}\in\C^{m-1}$.

\begin{thm}
\label{prop ind}
There exists a scalar polynomial
$\ga_{m}(\tilde c,c_m)$ on $\C^m$  such that the vector field
$\Ga_r = \Gamma_{r,{\tilde J}} + \ga_{m}({\tilde c},c_m)\frac \der{\der c_m}$
satisfies \Ref{formula main} for all $(\tilde c,c_m)\in\C^m$.
\end{thm}

The first statement of Theorem \ref{thm main} follows from Theorem \ref{prop ind}.

\subsection{ Proof of Theorem \ref{prop ind}}
\label{sec proof prop ind}

\begin{lem}
\label{lem ker 1} We have
\bean
\label{formula for ker 1}
d\frak m_i\big|_{\mu^J({\tilde c},c_m)} \left(\frac{\der}{\der t_r}\big|_{\mu^J({\tilde c},c_m)} -
\frac{\der\mu^J}{\der \Ga_{r,{\tilde J}}}({\tilde c},c_m)\right) =0
\eean
for all $i\neq j_m$ and $({\tilde c},c_m)\in\C^m$.

\end{lem}

\begin{proof}

By Theorem \ref{thm gaugemiura}, we have $\frak m_i \circ\mu^J({\tilde c},c_m)=\frak m_i \circ\mu^{{\tilde J}}({\tilde c})$ for
all $i\neq j_m$.
Hence,
\bean
\label{Ga inv}
d\frak m_i\big|_{\mu^J({\tilde c},c_m)} \left(\frac{\der\mu^J}{\der \Ga_{r,{\tilde J}}}({\tilde c},c_m)\right)
=
\frac{\der (\frak m_i \circ\mu^J)}{\der \Ga_{r,{\tilde J}}} ({\tilde c},c_m)=
 \frac{\der(\frak m_i \circ\mu^{{\tilde J}})}{\der \Ga_{r,{\tilde J}}} ({\tilde c}).
 \eean
By Theorems \ref{thm mkdvtokdv} and \ref{thm gaugemiura}, we have
\bean
\label{d/dt inv}
d\frak m_i\big|_{\mu^J({\tilde c},c_m)} \left(\frac{\der}{\der t_r}\big|_{\mu^J({\tilde c},c_m)} \right) =
\frac{\der}{\der t_r}\big|_{\frak m_i \circ\mu^J({\tilde c},c_m)} = \frac{\der}{\der t_r}\big|_{\frak m_i \circ\mu^{{\tilde J}}({\tilde c})}.
\eean
By the induction assumption, we have
\bean
\label{formula Ga t}
\frac{\der}{\der t_r}\big|_{\frak m_i \circ\mu^{{\tilde J}}({\tilde c})} =
\frac{\der(\frak m_i \circ\mu^{{\tilde J}})}{\der \Ga_{r,{\tilde J}}} ({\tilde c}).
\eean
These three formulas prove the lemma.
\end{proof}

The space $\mc M$ of Miura opers is an affine space. The tangent space to $\mc M$ at any point
is canonically identified with the vector space $\B(\text{Diag}_0)$, whose elements have the form
\linebreak
$X=\sum_{k=1}^N X_ke_{k,k}$ with $X_k\in\B$ and $\sum_{k=1}^NX_k=0$.

\begin{prop}
\label{lem ker dm}
Let $X(c)= \sum_{k=1}^N X_k(x,c)e_{k,k}$ be a $\B(\text{Diag}_0)$-valued function on $\C^m$ such that
\bean
\label{ ker X=0}
d\frak m_i\big|_{\mu^J(c)} (X(c))=0
\eean
 for all $i\neq j_m$ and  all $c\in\C^m$. Assume that $X_k=0$ unless $k=j_\ell$ of $k=j_{\ell}+1$
for some $\ell\leq m$. Then
\bean
\label{formula ker}
X \ = \ a\,\frac{\der \mu^J}{\der c_m}
\eean
for some scalar function $a$ on $\C^m$.
\end{prop}

\begin{proof}

Suppose $j_m = 1$.  For other values of $j_m$, the following argument
holds after a cyclic permutation of indices.

Let $\mu^J(c) = \der + \La + \sum_{k = 1}^N v_k(x,c) e_{k,k}$.
For $i=1,\dots,N$, the Miura map is
\bea
\frak m_i\ : \ \der + \Lambda + \sum_{k = 1}^N v_k e_{k,k}\ \mapsto\
(\der - v_i)(\der - v_{i-1})\dots(\der - v_1) (\der - v_N) \dots (\der - v_{i-1}).
\eea
For $i=2,\dots,N$, we have
$d\frak m_i\big|_{\mu^J(c)} (\sum_{k=1}^N X_k(x,c)) =  0.$
This means
\bean
\label{Xk}
&&
(-X_i)(\der - v_{i-1})\dots(\der - v_1) (\der - v_N) \dots (\der - v_{i-1})+
\\
\notag
&&
\phantom{aa}
+(\der - v_i)( - X_{i-1})\dots(\der - v_1) (\der - v_N) \dots (\der - v_{i-1})
+\dots+
\\
\notag
&&
\phantom{aaa}
+(\der - v_i)(\der - v_{i-1})\dots(- X_1) (\der - v_N) \dots (\der - v_{i-1})
+
\\
\notag
&&
\phantom{aaaa}(\der - v_i)(\der - v_{i-1})\dots(\der - v_1) (- X_N) \dots (\der - v_{i-1})+\dots+
\\
\notag
&&
\phantom{aaaaa}
+(\der - v_i)(\der - v_{i-1})\dots(\der - v_1) (\der - v_N) \dots ( - X_{i-1}) = 0.
\eean
For a given $i$, the left hand side is a differential operator of order $N-2$ and the right hand side is zero. Hence, for a given
$i$, we obtain $N-1$ linear differential equations on coordinates $X_1,\dots,X_N$. The index $i$ may take $N-1$ values $2,\dots,N$.
All together we get $(N-1)^2$ linear differential equations on $X_1,\dots,X_N$. Our goal is to show that under conditions
of Proposition \ref{lem ker dm} the space of solutions of this system of equations is at most one-dimensional and if the space is one-dimensional
then it is generated by
\bean
\label{ker formula}
\frac{y_{N}(x,{\tilde c}, m-1) y_{2}(x,{\tilde c},m-1)}
{y_{1}(x,{\tilde c},c_m,m)^2}\, H_{1},
\eean
see formula \Ref{form der c-m}.

The coefficient of $\der^{N-2}$ in equation \Ref{Xk} for $i=N,N-1,\dots,2$,
gives the following system of equations:
\bean
\label{N-2 coeff}
X'_{N-1} + 2X'_{N-2} + \dots + (N-1)X'_1 + \sum_{k=1}^N v_kX_k &=& 0,
\\
\notag
X'_{N-2} + 2X'_{N-3} + \dots + (N-2)X_1'+ (N-1)X'_N + \sum_{k=1}^N v_kX_k &=& 0,
\\
\notag
X'_{N-3} + 2X'_{N-4} + \dots + (N-3)X_1'+ (N-2)X'_N + (N-1)X_{N-1}+ \sum_{k=1}^N v_kX_k &=& 0,
\\
\notag
\dots \qquad \dots \dots \qquad \dots \qquad \dots \qquad \dots \qquad \qquad  && \phantom{aa}
\\
\notag
X'_{1} + 2X'_{N} + \qquad \dots  \qquad  \dots \qquad + (N-2)X_4'+ (N-1)X_3' + \sum_{k=1}^N v_kX_k &=& 0.
\eean
By subtracting the second equation from the first we get $(N-1)X_N' - X_1'-\dots-X_{N-1}'=0$. Since $\sum_{k=1}^NX_k=0$ we get
$X_N'=0$. Similarly we obtain
\bean
\label{N-2 one}
X_k'=0 \qquad \text{for} \quad k=3,\dots,N \qquad\text{and}\quad X_2'= - X_1'.
\eean
By using \Ref{N-2 one} and equation $X_2=\sum_{k\neq 2}X_k$ we obtain from \Ref{N-2 coeff}
the equation
\bean
\label{N-2 one}
(\der + v_1-v_2)X_1 = \sum_{k=3}^N(v_2-v_k)X_k.
\eean
If
\bean
\label{eqn zero sum}
\sum_{k=3}^N(v_2-v_k)X_k=0,
\eean
 then $(\der + v_1-v_2)X_1 = 0$ or
\bea
X_1\ =\ a\ \frac{y_{N}(x,{\tilde c}, m-1) y_{2}(x,{\tilde c},m-1)}
{y_{1}(x,{\tilde c},c_m,m)^2}
\eea
for some $a\in\C$. This equation proves Proposition \ref{lem ker dm}.

Now our goal is to show that under  assumptions of Proposition \ref{lem ker dm} equation \Ref{eqn zero sum}
holds.

\begin{lem}
\label{lem N-3}

The coefficient of $\der^{N-3}$ in equation \Ref{Xk} for $i=2,\dots,N$
gives the following system of equations:
\bean
\label{sln 2}
&&
\sum_{k=3}^i[(v_k-v_2)(v_1-v_k) + (k-2)v'_k-\sum_{j=2}^{k-1}v_j']\,X_k
+
\\
&&
\phantom{aaaaaa}
+ \sum_{k=i+1}^N[(v_k-v_2)(v_1-v_k) + (k-N-1)v'_k+v_1'+\sum_{j=k+1}^{N}v_j']\,X_k\,
=\, 0.
\notag
\eean
In particular, by subtracting  the consecutive equations we get equations
\bean
\label{sln 3}
v_k'\,X_k\,=\,0,
\qquad
k=3,\dots,N.
\eean
\end{lem}

\begin{proof}

Equate to zero the   coefficient of $D^{N-3}$ in \Ref{Xk}. We get
\bea
&&
\sum_{k=1}^N\,[a_{ik}\der^2 + b_{ik}\der +c_{ik}]\,X_k =
[a_{i1}\der^2 + b_{i1}\der +c_{i1}]\,X_1 +
[a_{i2}\der^2 + b_{i2}\der +c_{i2}]\,X_2 +
\phantom{aaaaaaa}
\\
&&
+ \sum_{k=3}^N\,c_{ik}\,X_k
= \ [(a_{i1}-a_{i2})\der^2 + (b_{i1}-b_{i2})\der +(c_{i1}-c_{i2})]\,X_1 +
\sum_{k=3}^N\,(c_{ik}-c_{i2})\,X_k
 =\ 0 .
\eea
The coefficient $c_{ik}$ has the form $c_{ik} = - q({k}) + d_{ik}$,
where
\bean
\label{formula q(k)}
q({k}) = \sum_{1\leq s<t\leq N\atop (s-k)(t-k)\neq 0} v_sv_t
\eean
 and $d_{ik}$ is given by
\bean
\label{d coe}
&&
d_{i\geq k}=v_{i-1}'+2v_{i-2}'+\dots +  (i-k-1)v_{k+1}'+ (i-k)v_{k-1}'+\dots \\
&&
\phantom{aaaaa}
+(i-3)v_2'+(i-2)v_1'
+(i-1)v_N' + \dots +(N-2)v_{i+1}',
\notag
\\
&&
d_{i<k}=v_{i-1}'+2v_{i-2}'+\dots  +(i-2)v_2'+(i-1)v_1'+
\notag
\\
&&
\phantom{aaaaa}
+iv_N' + \dots + (N+i-k-2)v_{k+1}'+ (N+i-k)v_{k-1}'+\dots +(N-2)v_{i+1}'.
\notag
\eean
It is easy to see that
\bean
\label{X1 coeff}
\phantom{aaa}
[(a_{i1}-a_{i2})\der^2 + (b_{i1}-b_{i2})\der +(c_{i1}-c_{i2})] =
(-(i-2)\der + \sum_{j=3}^Nv_j)(\der + v_1-v_2).
\eean
The formulas \Ref{N-2 one}, \Ref{formula q(k)}, \Ref{d coe}  , \Ref{X1 coeff} imply
Lemma \ref{lem N-3}.
\end{proof}

Let us finish the proof of Proposition \ref{lem ker dm}. Recall that we assumed that $j_m=1$ and it remains to be proved
that $\sum_{k=3}^N(v_2-v_k)X_k=0,$ see \Ref{eqn zero sum}. If $k$ is not equal to $j_\ell$ or $j_\ell+1$ for some $\ell\leq m$,
then $X_k=0$ by the assumption. Otherwise, we have $v'_kX_k=0$, see \Ref{sln 3}. For such $k$ the function $v_k$ is not identically zero
as a function on $\C^m$. By Lemma \ref{lem v' v} the function $v_k'$ is not identically zero on $\C^m$. Hence $X_k=0$ and equation
\Ref{eqn zero sum} holds.  Proposition \ref{lem ker dm} is proved.
\end{proof}

Denote
\bean
\label{ker vector}
Y = \frac{\der}{\der t_r}\big|_{\mu^J({\tilde c},c_m)} -
\frac{\der\mu^J}{\der \Ga_{r,{\tilde J}}}({\tilde c},c_m).
\eean
Let $Y_1,\dots, Y_N$ be coordinates of $Y$, \ $Y=\sum_{k=1}^NY_ke_{k,k}$.

\begin{lem}
\label{lem vector is good}
We have $Y_k=0$ unless $k=j_\ell$ or $j_\ell+1$ for some $\ell\leq m$.

\end{lem}

\begin{proof}
The fact that the $k$-th coordinate of $\frac{\der\mu^J}{\der \Ga_{r,{\tilde J}}}({\tilde c},c_m)$ equals zero unless
 $k=j_\ell$ or $j_\ell+1$ for some $\ell\leq m$ follows from \Ref{oper2}.
The fact that the $k$-th coordinate of $\frac{\der}{\der t_r}\big|_{\mu^J({\tilde c},c_m)} $ equals zero unless
 $k=j_\ell$ or $j_\ell+1$ for some $\ell\leq m$ follows from formulas \Ref{T mkdv}, \Ref{formula exp} and \Ref{formula La}.
\end{proof}

By Lemma \ref{lem vector is good} and Proposition \ref{lem ker dm} there exists a scalar function
$\ga_{m}(\tilde c,c_m)$ on $\C^m$  such that the vector field
$\Ga_r = \Gamma_{r,{\tilde J}} + \ga_{m}({\tilde c},c_m)\frac \der{\der c_m}$
satisfies \Ref{formula main} for all $(\tilde c,c_m)\in\C^m$. It remains to prove that the scalar function
is a polynomial.

\begin{prop}
\label{prop polyn}
The function $\ga_{m}({\tilde c},c_m)$ is a polynomial on $\C^m$.

\end{prop}

\begin{proof}

Let $g = x^d + \sum_{i=0}^{d-1} A_i(c_1,\dots,c_m) x^i$ be a polynomial in $x,c_1,\dots,c_m$. Denote
$h = \log'g$ the logarithmic derivative of $g$ with respect to $x$. Consider the Laurent expansion
of $h$ at $x=\infty$, \
$h = \sum_{i=1}^\infty B_i(c_1,\dots,c_k) x^{-i}$.

\begin{lem}
\label{ lem Laurent}
All coefficients $B_i$ are polynomials in $c_1,\dots,c_m$.
\qed
\end{lem}

The vector $Y = \frac{\der}{\der t_r}\big|_{\mu^J({\tilde c},c_m)}$ is a diagonal matrix depending
on $x,c_1,\dots,c_m$.
Let $Y_1,\dots, Y_N$ be coordinates of $Y$, \ $Y=\sum_{k=1}^NY_ke_{k,k}$.

\begin{lem}
\label{ lem  dt Laurent}
Each coordinate $Y_k$ is a rational function of $x,c_1,\dots,c_m$ which  has a Laurent expansion of the form
$Y_k = \sum_{i=1}^\infty B_i(c_1,\dots,c_k) x^{-i}$ where all coefficients $B_i$ are polynomials in $c_1,\dots,c_m$.
\qed
\end{lem}

The vector $Y = \frac{\der \mu^J}{ \Ga_{r,\tilde J} }({\tilde c},c_m)$ is a diagonal matrix depending
on $x,c_1,\dots,c_m$.
Let $Y_1,\dots, Y_N$ be coordinates of $Y$, \ $Y=\sum_{k=1}^NY_ke_{k,k}$.

\begin{lem}
\label{ lem  mu Laurent}
Each coordinate $Y_k$ is a rational function of $x,c_1,\dots,c_m$ which  has a Laurent expansion of the form
$Y_k = \sum_{i=1}^\infty B_i(c_1,\dots,c_k) x^{-i}$ where all coefficients $B_i$ are polynomials in $c_1,\dots,c_m$.
\qed
\end{lem}

Let us finish the proof of Proposition \ref{prop polyn}. The function  $\ga_{m}({\tilde c},c_m)$ is determined from the
equation
\bea
\frac{\der}{\der t_r}\big|_{\mu^J({\tilde c},c_m)} -
\frac{\der\mu^J}{\der \Ga_{r,{\tilde J}}}({\tilde c},c_m) = \ga_{m}({\tilde c},c_m)\frac{y_{N}(x,{\tilde c}, m-1) y_{2}(x,{\tilde c},m-1)}
{y_{1}(x,{\tilde c},c_m,m)^2}\, H_{1}.
\eea
The function
$\frac{y_{N}(x,{\tilde c}, m-1) y_{2}(x,{\tilde c},m-1)}
{y_{1}(x,{\tilde c},c_m,m)^2}$ has a Laurent expansion of the form $\sum_{i=1}^\infty B_i(c_1,\dots,c_k) x^{-i}$ and the first nonzero coefficient
$B_i$ of this expansion is 1 since the polynomials $y_N, y_1, y_2 $ are all monic. Hence $g_m$ is a polynomial.
\end{proof}
Theorem \ref{thm main} is proved.

\subsection{Cyclic generation}
\label{sec cyclic generation}

Recall that a sequence  $J=(j_1,\dots,j_m)$ with  $1\leq j_i \leq N$ is called cyclic if
$j_1=1$ and $j_{i+1} = j_i+1$ mod $N$ for all $i$.
If the generation is cyclic, then it is  degree increasing, see Lemma \ref{lem cyclic increas}.

\begin{thm}
\label{thm cyclic}
If $J$ is cyclic, then the image of the generation map $\mu^J:\C^m\to \mc M$
does not have a proper subset invariant under the the flows
$exp\left(\sum_{r \in \N} t_r \frac{\der}{\der t_r}\right)$.
\end{thm}


\begin{proof}
Let $\ell$ be a nonnegative number. Write $\ell=(N-1)p+q$ where $1\leq q\leq N$.
Define $r_\ell= \ell +p$.

\begin{lem}
\label{lem rm}

Let $J=(j_1,\dots,j_m)$ be a cyclic sequence.

\begin{enumerate}

\item[(i)]
Then the image of the map
$\mu^J :\C^m\to\mc M$ is point-wise fixed by the mKdV flow $\der_{t_r}$
for every $r>r_m$.

\item[(ii)] Let $\Ga_r$ be the vector fields on $\C^m$  described in
Theorem \ref{thm main}. Then $\Ga_{{r_m}} = a \frac\der{\der c_m}$ where $a$ is a nonzero number.

\end{enumerate}
\end{lem}

\begin{proof}
The lemma is a corollary of formula \Ref{T mkdv} and Lemmas \ref{lem exp}, \ref{lem lambda},
\ref{lem formula}.
\end{proof}

Theorem \ref{thm cyclic} follows from Lemma \ref{lem rm}.
\end{proof}

\section{Schur polynomials}
\label{sec SchuR poly}

\subsection{Definition}
By a {\it partition} we will mean an infinite sequence of nonnegative integers
$\la=(\la_0\geq\la_1\geq \dots)$ such that all except a finite number of the $\la_i$ are zero.
The number $|\la|=\sum_i\la_i$ is called the {\it weight of} $\la$.

Denote $t=(t_2,t_3,\dots).$
Define polynomials $h_i(t_1,t)$, $i=0,1,\dots$, by the relation
\bea
\text{exp}\,(-\sum_{j=1}^\infty t_jz^j)\,=\, \sum_{i=0}^\infty h_iz^i .
\eea
Set $h_i=0$ for $i<0$.
We have
\bean
\label{der h}
\frac {\der h_i}{\der t_1} = - h_{i-1}\, \qquad\text{and}\qquad
h_i(t_1,t=0) = \frac{(-t_1)^i}{i!} .
\eean
Define the Schur polynomial associated to a partition  $\la = (\la_0\geq\la_1\geq \dots\geq\la_n\geq \la_{n+1}=0)$ by
the formula
\bean
\label{def schur}
F_\la(t_1,t)\, =\, \text{det}_{i,j=0}^n\,(h_{\la_i-i+j}).
\eean

\begin{thm}[\cite{Fa}]
\label{thm irred}
For any $\la$, the Schur polynomial $F_\la$ is irreducible in
variables  $t_1, t$.
\end{thm}

\begin{cor}
\label{cor schur and der}
For a generic fixed $t$, the roots of $F_\la(t_1,t)$ with respect to $t_1$ are all simple.
\end{cor}

\begin{cor}
\label{cor two schur}

Let $\la, \mu$ be partitions, $\la\ne\mu$. Then for a generic fixed $t$, the polynomials
$F_\la(t_1,t)$ and $F_\mu(t_1,t)$ with respect to $t_1$ have no common roots.

\end{cor}

\subsection{Wronskian determinant}

For   $n\in\N$ and any functions $f_1,\dots,f_n$ of $t_1$ define the Wronskian determinant
\bean
\label{wr det}
\Wr_{t_1}(f_1,\dots,f_n) = \left(\begin{matrix}
f_1(t_1) & f_1'(t_1) &\dots & f_1^{(n-1)}(t_1) \\
f_2(t_1) & f_2'(t_1) &\dots & f_2^{(n-1)}(t_1)\\
 \dots & \dots &\dots & \dots \\
f_n(t_1) & f_n'(t_1) &\dots & f_n^{(n-1)}(t_1)
\end{matrix}\right),
\eean
where derivatives are taken with respect to $t_1$.

\begin{lem} [\cite{MV1}]
\label{lem MV}
For functions $f_1,\dots,f_n,g_1,g_2$ of $t_1$ we have
\bean
\label{wr MV}
&&
\phantom{aaaaaa}
\\
\notag
&&
\Wr_{t_1}(\Wr_{t_1}(f_1,\dots,f_n,g_1),\Wr_{t_1}(f_1,\dots,f_n,g_2)) = \Wr_{t_1}(f_1,\dots,f_n)\Wr_{t_1}(f_1,\dots,f_n,g_1,g_2).
\eean
\end{lem}

\begin{proof}
If $f_1,\dots,f_n,g_1$ are linearly dependent, then both sides of \Ref{wr MV} are equal to zero.
Assume that $f_1,\dots,f_n,g_1$ are linearly independent.

Both sides of \Ref{wr MV} depend on $g_2$ linearly and
are linear combinations of $g_2, g_2',$ \dots, $g_2^{(n+1)}$.
 Both sides are equal to zero if $g_2=f_1,\dots,f_n,g_1$.
In both sides, the coefficient of $g_2^{(n+1)}$ equals
$\Wr_{t_1}(f_1,\dots,f_n)\Wr_{t_1}(f_1,\dots,f_n,g_1).$ This implies the lemma.
\end{proof}

\begin{lem}
\label{lem der h}
For a partition $\la = (\la_0\geq\la_1\geq \dots\geq\la_n\geq \la_{n+1}=0)$, we have
\bean
\label{sch wr}
F_\la(t_1,t) = \Wr_{t_1}(h_{\la_0+n}, h_{\la_1+n-1}, \dots, h_{\la_n}).
\eean

\end{lem}

\begin{proof}
The lemma follows from \Ref{der h}.
\end{proof}

\begin{lem}
\label{deg t1}
For a partition $\la$, the Schur polynomial $F_\la(t_1,t)$ has degree $|\la|$ with respect to
$t_1$. The coefficient of $t_1^{|\la|}$ in $F_\la$ is a nonzero rational number.
\qed

\end{lem}

For a partition  $\la = (\la_0\geq\la_1\geq \dots\geq\la_n\geq \la_{n+1}=0)$ if $i\in\{0,\dots,n\}$
is such that $\la_i-1\geq\la_{i+1}$, then
$\la^i:=(\la_1\geq\dots\geq\la_{i-1}\geq\la_i-1\geq\la_{i+1}\geq\dots\geq\la_{n+1}=0)$ is a partition.
This partition will be called the {\it derivative partition of $\la$ at the $i$-th position}.

\begin{lem}
\label{lem der sCUr}
For a partition $\la = (\la_0\geq\la_1\geq \dots\geq\la_n\geq \la_{n+1}=0)$, we have
\bean
\label{sch dEr}
\frac {\der F_\la}{\der t_1} + \sum F_{\la^i} = 0
\eean
where the sum is over all derivative partitions $\la^i$ of the partition  $\la$.

\end{lem}

\begin{proof}
The lemma follows from Lemma \ref{lem der h} and formula  \Ref{der h}.
\end{proof}

\subsection{Subsets of virtual cardinal zero}
\label{Subsets of virtual cardinal zero}

Following \cite{SW}, we say that a subset $S=\{s_0<s_1<s_2<\dots\} \subset \Z$ is of {\it virtual cardinal zero},
if $s_i=i$ for all sufficiently large $i$. Denote $\mc S$ the set of all subsets of virtual cardinal zero.

\begin{lem} [\cite{SW}]
\label{lem SiW}
There is a one to one correspondence between elements of $\mc S$ and partitions, given by
$S \leftrightarrow\la$ where $\la_i =i-s_i$.

\end{lem}

For a subset $S=\{s_0<s_1<s_2<\dots\} \subset \Z$  and an integer $k\in \Z$ we denote
$S+k$ the subset $S=\{s_0+k<s_1+k<s_2+k<\dots\} \subset \Z$.

Let $S$ be a subset of virtual cardinal zero. Let $A=\{a_1,\dots,a_k\} \subset \Z$ be a finite
subset of distinct integers.

\begin{lem}
\label{lem shift}
If $\{a_1,\dots,a_k\} \cap (S+k) = \emptyset$. Then
$\{a_1,\dots,a_k\} \cup (S+k)$ is a subset of virtual cardinal zero.
\qed

\end{lem}

To a subset $S=\{s_0<s_1<s_2<\dots< s_{n+1}=n+1<\dots\}$ of virtual cardinal zero we assign the Schur polynomial of the
partition $\la$ given by Lemma \ref{lem SiW},
\bean
\label{sch of virt}
F_S(t_1,t) : = F_\la(t_1,t) = \text{det}_{i,j=0}^n (h_{j-s_i}) = \Wr_{t_1}(h_{n-s_0},\dots, h_{n-s_n}).
\eean

Let $S$ be a subset of virtual cardinal zero. Assume that the two-element subset $A=\{a_1<a_2\}\subset \Z$ is such that
$A\cap (S+1)=\emptyset$. Then the subsets $S_1:=\{a_1\}\cup (S+1), S_2:=\{a_2\}\cup (S+1), S_3:=S,
S_4=\{a_1+1, a_2+1\}\cup (S+2)$ are of virtual cardinal zero, by Lemma \ref {lem shift}.

\begin{thm}
\label{thm new identity}
 We have
\bean
\label{formula new}
\Wr_{t_1}(F_{S_1}, F_{S_2})=F_{S_3}F_{S_4} .
\eean
\end{thm}
 \begin{proof}
Let $S=\{s_0<s_1<s_2<\dots< s_{n+1}=n+1 <\dots\}$. We have
\bea
&&
F_{S_1} = \Wr_{t_1}(h_{n-s_0}, \dots, h_{n+1-a_1}, \dots, h_{n-s_n}),
\\
&&
F_{S_2} = \Wr_{t_1}(h_{n-s_0}, \dots, h_{n+1-a_2}, \dots, h_{n-s_n}),
\\
&&
F_{S_4} = \Wr_{t_1}(h_{n-s_0}, \dots, h_{n+1-a_1}, \dots, h_{n+1-a_2}, \dots, h_{n-s_n}).
\eea
Now the theorem follows from Lemma \ref{lem MV}.
 \end{proof}

 As was explained to us by A. Lascoux, this simple Theorem \ref{thm new identity}
 can be interpreted as an instance
of the Pl\"ucker relation [1; 2][3; 4] - [1; 3][2; 4] + [1; 4][2; 3] = 0, see Section 1.4 and
Complement A1 in \cite{L1}, cf. \cite{L2}.

\subsection{More  general Wronskian identities for Schur polynomials}
\label{sec more general}

Lemma \ref{lem MV} is a particular case of more general Wronskian identities
in \cite{MV1}. Namely, fix integers $0\leq k\leq s$ and functions $g_1,\dots,g_{s+1}$ of $t_1$.
Let
\bea
V_{s-k+1}(i)=\Wr_{t_1}(g_1,\dots,g_{s-k},g_i).
\eea

\begin{lem}[\cite{MV1}]
\label{MV lem1}
We have
\bean
\label{WR1}
&
\Wr_{t_1}(V_{s-k+1}(s-k+1),V_{s-k+1}(s-k+2), \dots, V_{s-k+1}(s+1)) =&
\\
&
\notag
=(\Wr_{t_1}(g_1,\dots,g_{s-k}))^k
\Wr_{t_1}(g_1,\dots,g_{s+1}).
&
\eean
\end{lem}

Let $W_s(i) = \Wr_{t_1}(g_1,\dots,\widehat{g_i},\dots,g_{s+1})$ be the Wronskian of all functions except
$g_i$.

\begin{lem}[\cite{MV1}]
\label{MV lem2} We have
\bean
\label{WR2}
&
\Wr_{t_1}(W_{s}(s-k+1),W_{s}(s-k+2), \dots, W_{s}(s+1)) =
&
\\
&
\notag
=\Wr_{t_1}(g_1,\dots,g_{s-k})
(\Wr_{t_1}(g_1,\dots,g_{s+1}))^k.
&
\eean
\end{lem}

If $\{g_1,\dots,g_{s+1}\}$ is an arbitrary subset of the set
$\{h_1,h_2,\dots,\}$, then Lemma \ref{lem der h} and Lemmas \ref{MV lem1}, \ref{MV lem2}
give identities for Schur polynomials. For example if $k=2,s=3$ and
$\{g_1,g_2,g_3,g_4\}=\{h_4,h_3,h_2,h_1\}$, we get:
\bean
\label{EXAMPLE}
&&
\Wr_{t_1}(F_{(3,3)},F_{(3,2)},F_{(3,1)})=F_{(4)}^2F_{(1,1,1,1)},
\\
\notag
&&
\Wr_{t_1}(F_{(2,1,1)},F_{(2,2,1)},F_{(2,2,2)})=F_{(4)}F_{(1,1,1,1)}^2.
\eean

\subsection{KdV subsets}
\label{Sec KdV subsets}

Fix an integer $N>1$. We say that a subset $S$ of virtual cardinal zero is a
{\it KdV subset} if
$S+N\subset S$. For example, for any $N>1$
\bea
S^\emptyset = \{0,1,2,\dots\}
\eea
 is a KdV subset.

\begin{lem}
\label{lem leading term}
Let $S$ be a  KdV subset. Then there exists a unique $N$-element subset
$A=\{a_1< \dots <a_N\}\subset \Z$ such that
$S = A\cup (S+N)$.
\qed
\end{lem}

The subset $A$ of the Lemma \ref{lem leading term} will be called the
{\it  leading term of}  $S$.

\begin{lem}
\label{lem Lead}
The leading term $A$ uniquely determines the KdV subset $S$, since
 $S$  is the union of $N$ nonintersecting arithmetic progressions
$\{a_i, a_i+N, a_i + 2N, \dots\}$, $i=1,\dots,N$.
If $A=\{a_1<\dots<a_N\}\subset \Z$ is the leading term of a KdV subset $S$, then
$a_i-a_j$ are not divisible by $N$ for all
 $i\ne j$.
\qed
\end{lem}

For example, for $N=3$ and $S=\{-3<0< 1<3<4<\dots\}$, the leading term is
$A=\{-3,1,5\}$ and $S$ is the union of three arithmetic progressions with step 3:\
$\{-3,,0,3, 6, \dots\}$, $\{1,4,7,\dots\}$, and $\{5,8,11,\dots\}$.

\begin{lem}
\label{lem lead a-a+N}
Let a KdV subset $S$ has the leading term
 $A=\{a,a+1,\dots,a+N-1\}$ for some integer $a$. Then $a=0$ and
 $S=S^\emptyset$.
 \qed

\end{lem}

\begin{lem}
\label{lem mut kdv}
Let $S$ be a KdV subset with leading term $A$. For any $a\in A$ the subset
\bean
\label{j mut}
S[a] = \{a+1-N\} \cup (S+1)
\eean
is a KdV subset with leading term $A[a] = (A+1) \cup \{a + 1-N\} -\{a+1\}$
\qed
\end{lem}

The subset $S[a]$ will be called the {\it mutation of the KdV subset} $S$ at $a\in A$.

\medskip

For example, let $N=3$. For a KdV subset $S=\{-3<0< 1<3<4<\dots\}$
with leading term $A=\{-3<1 <5\}$, we have
$A[1] = \{-2<-1< 6\}$ and $S[1]=\{-2<-1<1<2<4<5<6<\dots\}$.

\begin{lem}
\label{lem two step}
Let $S_1$ be a  KdV subset with leading term $A$. Let $S_2$ be a KdV subset such that $S_1+1\subset S_2$.
Then $S_2$ is the mutation of $S_1$ at some element $a\in A$.
\qed

\end{lem}

\begin{thm}
\label{thm generation kDv}
Any KdV subset $S$ can be transformed to the KdV subset $S^\emptyset$ by a sequence of mutations.
\end{thm}

\begin{proof}
Let $A=\{a_1<\dots<a_N\}$ be the leading term of $S$. The number $|S|=a_N-a_1$ will be called the {\it width of} $S$.
The width is not less than $N-1$. If the width is $N-1$, then
$S=S^\emptyset$ by Lemma \ref{lem lead a-a+N}.

Assume that $|S|>N-1$. Then $|S[a_N]|<|S|$.
Repeating this procedure we will make the width to be equal to $N-1$.
\end{proof}

\begin{cor}
If  $A=\{a_1<\dots<a_N\}$  is the leading term of a KdV subset, then
\bean
\label{sum lead}
\sum_{i=1}^N a_i = \frac {N(N-1)}2.
\eean
\end{cor}

\begin{proof} The leading term of $S^\emptyset$ has this property and mutations do not change the sum
of the elements of the leading term.
\end{proof}

\begin{thm}
\label{thm char lead}
A subset $A=\{a_1<\dots<a_N\}$ is the leading term of a KdV subset if and only if
equation \Ref{sum lead} holds and $a_i-a_j$ is not divisible by $N$ for any $i\ne j$.

\end{thm}
\begin{proof}
The proof is similar to the proof of Theorem
\ref {thm generation kDv}.
\end{proof}

\subsection{mKdV tuples of subsets}
\label{sec mkdv tuples of subsets}

 We say that an $N$-tuple $\bs S = (S_1,\dots,S_N)$ of KdV subsets is an {\it mKdV tuple of subsets}
 if $S_i+1\subset S_{i+1}$ for all $i$, in particular, $S_N+1\subset S_1$.

For example, for any $N$
\bea
\bs S^\emptyset = (S^\emptyset,\dots,S^\emptyset)
\eea
is an mKdV tuple of subsets.

\begin{lem}
\label{lem permut}
If  $\bs S = (S_1,\dots,S_N)$  is an
mKdV tuple, then for any $i$, $(S_i,S_{i+1},\dots,S_N,$ $ S_1,S_2,\dots,S_{i-1})$
is an mKdV tuple of subsets.
\qed
\end{lem}

\medskip

Let $S$ be a  KdV subset with leading term $A=\{a_1<\dots<a_N\}$. Let $\sigma$ be
an element of the permutation group $\Si_N$. Define an $N$-tuple $\bs S_{S,\sigma} = (S_1,\dots,S_N)$
where
\bean
\label{mkdv tuple}
S_i = \{a_{\sigma(1)}+i-N, a_{\sigma(2)}+i-N, \dots, a_{\sigma(i)}+i-N\} + (S+i),
\eean
in particular, $S_N = A \cup (S+N) = S$.

\begin{lem}
\label{lem mkdv constr}
The $N$-tuple $\bs S_{S,\sigma}$ is an mKdV tuple.
\qed
\end{lem}

For example for $S=S^\emptyset$, $\bs S= \bs S^\emptyset$, we have
$\bs S^\emptyset =\bs S_{S^\emptyset, \sigma}$, where
$\sigma=(N,N-1,\dots,2,1)$ is the longest permutation.

\begin{thm}
\label{thm mkdv description}
Every  mKdV tuple is of the form $\bs S_{S,\sigma}$ for some KdV subset $S$ and some element
$\sigma\in\Si_N$.

\end{thm}

\begin{proof}
The theorem follows from Lemma \ref{lem two step}.
\end{proof}

\begin{cor}
For any KdV subset $S$ there exists exactly $N!$ mKdV tuples $\bs S=(S_1,\dots,S_N)$ such that
$S_N=S$.
\end{cor}

\subsection{Mutations of  mKdV tuples }
\label{Mutations of  mKdV tuples }
Let $\bs S=(S_1,\dots,S_N)$ be an mKdV tuple. By Theorem \ref{thm mkdv description},
we have $\bs S = \bs S_{S_N, \sigma}$ for some
 permutation $\sigma\in\Si_N$. Let $A=\{a_1,\dots,a_N\}$ be the leading term
of $S_N$. Then $S_i$ are given by formula \Ref{mkdv tuple}.

\begin{lem}
\label {lem constr}
For any $i=1,\dots,N$, there exists a unique mKdV tuple
\bean
\label{new mkdv tuple}
\bs S^{(i)}=(S_1,\dots,S_{i-1},\tilde S_i, S_{i+1},\dots,S_N)
\eean
which differs from $\bs S$ at the $i$-th position only.

\end{lem}

The mKdV tuple $\bs S^{(i)}$ will be called the {\it mutation of the mKdV tuple $\bs S$ at the $i$-th position}.
Denote $w_i : \bs S \mapsto \bs S^{(i)}$ the mutation map.

\begin{proof} First assume that $i<N$.  Let $s_{i,i+1}\in\Si_N$ be the transposition
of $i$ and $i+1$. Then
\bean
\label{oermU}
&&
(\sigma\circ s_{i,i+1}(1),\dots,\sigma\circ s_{i,i+1}(N)) =
\\
\notag
&&
\phantom{aaaa}
= (\sigma(1),\dots,\sigma(i-1), \sigma(i+1), \sigma(i), \sigma(i+2), \dots, \sigma(N)).
\eean
Set $\bs S^{(i)} = \bs S_{S, \sigma\circ s_{i,i+1}}$.  By formulas \Ref{oermU} and \Ref{mkdv tuple}
the mKdV tuple $\bs S^{(i)}$ differs from $\bs S$  at the $i$-th position only and the $i$-th term of
$\bs S^{(i)}$ is
\bean
\label{Tilde}
&&
\\
&&
\notag
\tilde S_i = \{a_{\sigma(1)}+i-N, a_{\sigma(2)}+i-N, \dots, a_{\sigma(i-1)}+i-N, a_{\sigma(i+1)}+i-N\} + (S+i).
\eean
The fact that $\bs S^{(i)}$ is unique follows from Theorem \ref {thm mkdv description}.

In order to mutate $\bs S$ at the $N$-th position, we consider the mKdV tuple $(S_N,S_1,\dots,S_{N-1})$, see Lemma \ref{lem permut},
and then mutate this tuple at the first position.
\end{proof}

\begin{lem}
\label{lem wrO}
Let
$\bs S=(S_1,\dots,S_N)$ be an mKdV tuple. For any $i$ let
$\bs S^{(i)}=(S_1,\dots,S_{i-1},$ $\tilde S_i, S_{i+1},\dots,S_N)$ be the mutation at the $i$-th position.
Then the four KdV subsets $S_{i}, $ $\tilde S_i,$ $ S_{i-1}, S_{i+1}$ satisfy the conditions of Theorem \ref{thm new identity}
and
\bean
\label{mKDv schur}
\Wr_{t_1}(F_{S_i},F_{\tilde S_i})\, =\, \pm\, F_{S_{i-1}}F_{S_{i+1}}.
\eean
\qed
\end{lem}

Let   $\la^{i}, $ $\tilde \la^i$  be the partitions corresponding to the KdV subsets
$S_{i}, $ $\tilde S_i$, respectively. Recall that
$\deg_{t_1} F_{S_i} =|\la^i|$ and $\deg_{t_1} F_{\tilde S_i}=|\tilde \la^i|$.

The mutation $w_i : \bs S \mapsto \bs S^{(i)}$ will be called {\it degree decreasing} if
$\deg_{t_1} F_{\tilde S_i} <\deg_{t_1} F_{S_i}$.

\begin{thm}
\label{thm mutation all}

Any mKdV tuple $\bs S$ can be transformed to the mKdV tuple $\bs S^\emptyset = (S^\emptyset,\dots,S^\emptyset)$
by a sequence of degree decreasing mutations.
\end{thm}

\begin{proof}

Let $\bs S=(S_1,\dots,S_N)$. Let $s_0^i$ be the minimal element in $S_i$. Denote
$s_{\min}=\min(s^1_0,$ $\dots,s_0^N)$.
Denote $d$ the number of $i$'s such that $s^i_0=s_{\min}$.

\begin{lem}
\label{lem d<N}
If $d=N$, then $\bs S = \bs S^\emptyset$.
\qed
\end{lem}

Thus we may assume that $d<N$. Now we will describe a degree decreasing mutation of $\bs S$ which will decreases $d$ by one
if $d>1$  and which will increase
$s_{\min}$ if $d=1$. That will prove the theorem.

Denote $S=S_N$. Then $\bs S=\bs S_{S,\sigma}$ for some $\sigma \in \Si_N$. Let
$A=\{a_1<\dots<a_N\}$ be the leading term of $S$.

\begin{lem}
\label{lem min}
We have
\bean
\label{miN}
s_{\text{min}}= \min\, \{a_{\sigma(1)}+1-N, a_{\sigma(2)}+2-N,\dots,a_{\sigma(N)}+N-N\} .
\eean
\qed

\end{lem}

Since $d<N$ we may choose $i$ such that $a_{\sigma(i)} + i-N=s_{\text{min}} < a_{\sigma(i+1)} + i+1-N$.
Then the mutation $w_i: \bs S\mapsto \bs S^{(i)}$ decreases $d$ by one
if $d>1$  and increases
$s_{\min}$ if $d=1$. The theorem is proved.
\end{proof}

\begin{cor}
\label{cor deg decrea}
Let $\bs S=(S_1,\dots,S_N)$ be an mKdV tuple of subsets such that $\bs S\ne \bs S^\emptyset$.
Let $(\la^1,\dots,\la^N)$ be the corresponding partitions.
Then there exists $i\in\{1,\dots,N\}$ such that $2|\la^i|>|\la^{i+1}|+|\la^{i-1}|+1$.
\qed
\end{cor}

For example, let $N=3$ and  $\bs S=(S_1,S_2,S_3)$ with
\bean
\label{eX}
&&
\phantom{aaaa}
\\
\notag
&&
S_1=\{-1<0<2<3<\dots\},\quad
 S_2=\{-2, 0,1,3,4,\dots\},
 \quad S_3=\{-1<1<2<\dots\}.
 \eean
Then  $s_{\text{min}} = -2$ and the mutation $w_2$ produces the triple
$\bs S^{(2)}=(S_1,\tilde S_2,S_3)$, $\tilde S_2=\{0<1<\dots \}$
and decreases $s_{\text{min}}$.

\subsection{mKdV tuples and critical points}

Let $\bs S=(S_1,\dots,S_N)$ be an mKdV tuple.
Let $(\la^1,\dots,\la^N)$ be the $N$-tuple of the corresponding partitions provided by Lemma \ref{lem SiW}.
 Let $(F_{\la^1}(t_1,t),\dots,$ $ F_{\la^N}(t_1,t))$ be the corresponding tuple
of Schur polynomials.

\begin{thm}
\label{thm schur-crit}
For a generic fixed $t$, the tuple $(F_{\la^1}(x,t),\dots, F_{\la^N}(x,t))$ of polynomials in $x$ represents
a critical point of the master function $\Phi$ given by formula \Ref{Master} where the numbers $k_1,\dots,k_N$
are the numbers $|\la^1|$,\dots, $|\la^N|$.
\end{thm}

For example, if $N=3$ and $\bs S=(S_1,S_2,S_3)$ is given by \Ref{eX},
then $\la^1=(1,1), \la^2=(2,1,1), \la^3=(1)$ is the corresponding triple of partitions.
For a generic fixed $t$ the triple
$(F_{\la^1}(x,t), F_{\la^2}(x,t), F_{\la^3}(x,t))$ represents a critical point
of the master function with 7 variables $(u^{(j)}_i)$, $j=1,2,3$, $i=1,\dots,k_j$, with  $(k_1,k_2,k_3)= (2,4,1)$.

\medskip
\noindent
{\it Proof of Theorem \ref{thm schur-crit}.}
For generic fixed $t$ and any $i$, all roots of $F_{\la^i}(x,t)$ are simple and the polynomials
$F_{\la^i}(x,t), F_{\la^{i+1}}(x,t)$ do not have common roots, see Corollaries
 \ref{cor schur and der} and \ref{cor two schur}. For any fixed $t$ the tuple  $(F_{\la^1}(x,t),\dots, F_{\la^N}(x,t))$
 is fertile by Lemma \ref{lem wrO}. Now the theorem follows from
Theorem \ref{fertile cor}.
\qed

\subsection{mKdV tuples and the generation of critical points}
Let $\bs S =(S_1,\dots,$ $S_N)$ be an mKdV tuple. Let
$(F_{S_1}(x,t), \dots,F_{S_N}(x,t))$ be the corresponding tuple of
Schur polynomials. By Lemma \ref{deg t1}, there exist numbers $\al_1,\dots,\al_N$
such that
\bea
y_{\bs S}(x,t)=(\al_1F_{S_1}(x,t), \dots,\al_NF_{S_N}(x,t))
\eea
 is a tuple of monic polynomials in $x$ depending on parameters $t$.

 Recall that $t=(t_2,t_3,\dots)$ is an infinite sequence of parameters
but the tuple $y_{\bs S}(x,t)$ depends only on finitely many of them, say on $t_{\bs S}=
(t_2,t_3,\dots,t_{d+1})$ for some integer $d$. We introduce the map
\bean
\label{schur map}
Y_{\bs S} : \C^{d} \to (\C[x])^N,\qquad
t_{\bs S} \mapsto y_{\bs S}(x,t_{\bs S}).
\eean

By Theorem \ref{thm mutation all}, there exists a sequence $J=(j_1,\dots,j_m)$,
 $1\leq j_\ell\leq N$, such that $w_{j_1}w_{j_{2}}\dots w_{j_m} : \bs S\mapsto \bs S^\emptyset$
and each mutation $w_{j_\ell} : w_{j_{\ell + 1}}\dots w_{j_{m}} \bs S \mapsto w_{j_{\ell}}\dots w_{j_m} \bs S$
is degree decreasing.

Recall the map $Y^J: \C^m \to (\C[x])^N$, the generation of $N$-tuples of polynomials in $x$
from the tuple $y^\emptyset=(1,\dots,1)$ in the $J$-th direction,
defined in Section \ref{sec generation procedure}.

\begin{thm}
\label{thm schur induced}
The map $Y_{\bs S}$ can be induced from the map $Y^J$ by a suitable polynomial map
$f: \C^d \to\C^m$, that is $Y_{\bs S}= Y^J\circ f$.
\end{thm}

Theorem \ref {thm schur induced} says that every family of critical points provided my Theorem \ref{thm schur-crit}
appears as a subfamily of critical points generated from the tuple $y^\emptyset$.

\begin{proof}
The proof is by induction on $m$. For $m=0$ we have
$\bs S = \bs S^\emptyset$, $y_{\bs S}=(1,\dots,1)$,
$Y_{\bs S} : \C^d \to (\C[x])^N$ is the constant map
$t_{\bs S} \mapsto (1,\dots,1)$. The generation map $Y^\emptyset : \C^0 \to (\C[x])^N$ is the map $(pt)\mapsto
(1,\dots,1)$. The map $f: \C^d\to (pt)$ has the property required in the theorem.

Now assume that the theorem is proved for the sequence $\tilde J=(j_1,\dots,j_{m-1})$.
We have the maps
\bea
&
Y^{\tilde J} : \C^{m-1} \to (\C[x])^N,
&
\quad
\tilde c \mapsto (y_1(x,\tilde c), \dots, y_N(x,\tilde c)),
\\
&
Y^{J} : \C^{m} \to (\C[x])^N,
&
\quad
(\tilde c, c_m) \mapsto (y_1(x,\tilde c), \dots, y_{j_m,0}(x,\tilde c) + c_m y_{j_m}(x,\tilde c), \dots, y_N(x,\tilde c)),
\eea
see formulas \Ref{tilde y}, \Ref{Ja}. Recall that $y_{j_m,0}$ is the monic polynomial in $x$
satisfying the equation
\bean
\label{eq 1}
\Wr_x(y_{j_m},  y_{j_m,0})=\,\on{const}\,y_{j_m-1}y_{j_m+1}
\eean
 and such that the coefficient of $x^{\deg y_{j_m}}$ in $\tilde y_{j_m,0}$ equals zero, see Section \ref{Degree increasing generation}.

 We also have the mKdV tuples $\bs S=(S_1,\dots,S_N)$, $\bs S^{(j_m)} = (S_1,\dots, \tilde S_{j_m},\dots,S_N)$ and
 the corresponding  tuples of monic polynomials in $x$ depending on parameters $t_S=(t_2,\dots,t_{d+1})$, namely,
\bea
(\al_1F_{S_1}(x,t_{\bs S}),\dots,\al_NF_{S_N}(x,t_{\bs S})), \quad
(\al_1F_{S_1}(x,t_{\bs S}),\dots,\tilde \al_{j_m}F_{\tilde S_{j_m}}(x,t_{\bs S}), \dots,
\al_NF_{S_N}(x,t_{\bs S})).
\eea
We know that
\bean
\label{eq 2}
\Wr_x(F_{S_{j_m}}, F_{\tilde S_{j_m}})\,=\,\on{const}\, F_{S_{j_m-1}}F_{S_{j_m+1}}.
\eean
We have the maps
\bea
&
Y_{\bs S} : \C^d \to (\C[x])^N,
&
\quad
t_{\bs S} \mapsto (\al_1F_{S_1}(x,t_{\bs S}),\dots,\al_NF_{S_N}(x,t_{\bs S})),
\\
&
Y_{\bs S^{(j_m)}} : \C^d \to (\C[x])^N,
&
\quad
t_{\bs S} \mapsto
(\al_1F_{S_1}(x,t_{\bs S}),\dots,\tilde \al_{j_m}F_{\tilde S_{j_m}}(x,t_{\bs S}), \dots,
\al_NF_{S_N}(x,t_{\bs S})).
\eea
By induction assumptions we have a polynomial map $\tilde f : \C^d \to \C^{m-1}$ such that
 $Y_{\bs S^{(j_m)}} = Y^{\tilde J}\circ \tilde f$.

From formulas \Ref{eq 1}, \Ref{eq 2}, it is clear that there exists a unique scalar polynomial $f_{j_m}(t_{\bs S})$
such that
\bea
\al_{j_m}F_{S_{j_m}}(x,t_{\bs S}) = y_{j_m,0}(x,\tilde f(y_{\bs S})) + f_{j_m}(t_{\bs S})
\tilde \al_{j_m}F_{\tilde S_{j_m}}(x,t_{\bs S}) .
\eea
Then the map $f : \C^d\to \C^m, y_{\bs S}\mapsto (\tilde f(y_{\bs S}),  f_{j_m}(t_{\bs S})) $ has the required property.
The theorem is proved.
\end{proof}

\subsection{mKdV tuples and differential operators}
Let $S$ be a KdV subset with the leading term $A=\{a_1<\dots<a_N\}$.
Recall the mutations $S[a]$, $a\in A$. These are KdV subsets defined in \Ref{j mut}.

For  a permutation
$\sigma \in \Si_N$ consider the mKdV tuple $\bs S_{S, \sigma}=(S_1,\dots,S_N)$.
Recall that $S_N=S$. Consider
 the corresponding tuple of the Schur polynomials $(F_{S_1}(x,t),\dots,F_{S_N}(x,t))$ and
the differential operator
\bean
\label{diff oper S}
&&
\\
\notag
&&
\D_S = \left(\!\frac d{dx} - \log'\!\left(\!\frac{F_{S_N}(x,t)}{F_{S_{N-1}}(x,t)}\!\right)\!\right)\!\!
\left(\!\frac d{dx} - \log'\!\left(\!\frac{F_{S_{N-1}}(x,t)}{F_{S_{N-2}}(x,t)}\!\right)\!\right)\!
\dots
\!\left(\!\frac d{dx} - \log'\!\left(\!\frac{F_{S_{1}}(x,t)}{F_{S_{N}}(x,t)}\!\right)\!\right)
\eean
with respect to $x$. Here  $()'$ denotes the derivative with respect to $x$.
The differential operator depends on $t$ as a parameter.

\begin{thm}
\label{thm diff oper SCur}
The differential operator $\D_S$ does not depend on the permutation $\sigma\in\Si_N$.
For every fixed $t$, the rational functions $F_{S[a]}(x,t)/F_{S}(x,t),$ $ a\in A$, form a basis
of the kernel of $\D_S$.
\end{thm}

\begin{proof}
The first statement of the theorem is a direct corollary of  Lemma 5.2 in \cite{MV1}, Theorem 5.3 in \cite{MV1}
and Lemma \ref {lem wrO} in Section \ref{Mutations of  mKdV tuples }. The second statement follows from Lemma 5.6 in
\cite{MV1}.
\end{proof}

\begin{example}
For $N=3$, consider the KdV subset $S=\{-1<0<1<2<\dots\}$ with leading term $A=\{-1<0<4\}$.
 Consider the mKdV triple $\bs S_{S,\sigma}=(S_1,S_2,S_3)$ with $(\sigma(1),\sigma(2),\sigma(3))=(2,3,1)$.
Here  $S_1=\{-2<0<1<3<\dots\}$,
 $S_2=\{-1<1<2<\dots\}$,  $S_3=\{-1<0<1<2<\dots\}$. We have
 $S[-1]=\{-3<0<1<3<\dots\}$,  $S[0]=\{-2<0<1<3<\dots\}$,  $S[4]=\{0<1<\dots\}$.
 Then the functions $F_{S[a]}/F_S, a\in \{-1<0<4\}$, form a basis of the kernel of the differential operator
 \bean
 \label{eX}
 \D_S = \left(\frac d{dx} - \log'\left(\frac{F_{S_3}}{F_{S_{2}}}\right)\right)
\left(\frac d{dx} - \log'\left(\frac{F_{S_{2}}}{F_{S_{1}}}\right)\right)
\left(\frac d{dx} - \log'\left(\frac{F_{S_{1}}}{F_{S_{3}}}\right)\right).
\eean
 \end{example}

\subsection{Identities for Schur functions} In Section \ref{sec SchuR poly} we proved
identities relating the  Schur polynomials and their derivatives with respect to $t_1$, see
formula \Ref{formula new}, Theorem \ref{thm diff oper SCur}, Section \ref{sec more general}. By formula \Ref{sch dEr}, the
derivatives of a Schur polynomial can be expressed as a linear combination with integer coefficients
of Schur polynomials. Thus, each of those identities can be written as a polynomial relation
with integer coefficients for Schur functions. Since the Schur polynomials are the characters of representations
of the general linear group, each of such identities can be interpreted as an isomorphism of two representations.

For example, we have  $W_{t_1}(F_{(2,1)}, F_{(0)}) = F_{(1)}F_{(1)}$ by formula \Ref{sch dEr}.
This identity can be written as
\bea
F_{(2)} + F_{(1,1)}=F_{(1)}F_{(1)},
\eea
by formula \Ref{sch dEr} and can be interpreted as the statement that the tensor square of the representation
with highest weight $(1)$ is the direct sum of two representations with highest weights $(2)$ and $(1,1)$.

\subsection{Mutations of mKdV tuples and the affine Weyl group $\widehat {W}_{A_{N-1}}$}
\label{sec Mutations of mKdV tuples and the affine Weyl group }
The affine Weyl group $\widehat {W}_{A_{1}}$ is generated by elements $w_1,w_2$ subject to the relations
$w_1^2=w_2^2=1$.

For $N>2$, the affine Weyl group $\widehat {W}_{A_{N-1}}$ is generated by elements
$w_1,\dots,w_N$. We consider the indices of the generators modulo $N$, in particular, $w_0:=w_N$ and
$w_{N+1}:=w_1$. The relations are  $w_i^2=1$, $w_iw_{i+1}w_i=w_{i+1}w_iw_{i+1}$ for all $i$ and
$w_iw_j=w_jw_i$ otherwise.

For $N>1$, denote $\mc S_{mKdV}$ the set of all mKdV $N$-tuples.
In Section \ref{Mutations of  mKdV tuples } we defined the mutation maps
$w_i : \mc S_{mKdV} \to \mc S_{mKdV}$ for $i=1,\dots,N$.

\begin{thm}
\label{lem weyl action}
The mutation maps define a transitive action of the  Weyl group $\widehat{W}_{A_{N-1}}$ on the set
$\mc S_{mKdV}$.
\qed
\end{thm}

\begin{proof}
According to the proof of Lemma \ref{lem constr}, the mutation at the $i$-th position corresponds to
the transposition $s_{i,i+1}\in\Si_N$, see formula \Ref{oermU}. Thus the relation 
$w_iw_{i+1}w_i=w_{i+1}w_iw_{i+1}$ follows from the relation 
$s_{i,i+1}s_{i+1,i+2}s_{i,i+1}=s_{i+1,i+2}s_{i,i+1}s_{i+1,i+2}$ in the symmetric group, and so on.
The transitivity follows from Theorem \ref{thm mutation all}.
\end{proof}

\section{Tau-functions and critical points}
\label{Critical points and tau functions}

\subsection{Subspaces of $H$}
\label{SUBS}
For a Laurent polynomial
\bean
\label{eqn wsum}
v = \sum_i  v_{i} z^i ,
\eean
the number
 $\ord \,v \,=\, \text{min} \{ i : v_{i} \neq 0 \}$ will be called
the {\it order} of $v$.

Following \cite{SW},
let $H$ be the Hilbert  space $L^2(S^1)$ with
natural orthonormal basis $\{z^j\}_{j\in \Z}$.  Let $H_+$ be the closure of
the span of $\{z^j\}_{j \geq 0}$ and $H_-$ the closure of the span of
$\{z^j\}_{j < 0}$. We have the orthogonal decomposition $H=H_+\oplus H_-$.

Denote $\GR$ the set of all  closed subspaces $W \subset H$ such that
\bean
\label{-k k}
z^q H_+ \subset W \subset z^{-q} H_+
\eean
for some $q > 0$. Such subspaces can be identified with subspaces
$W/z^qH_+$ of
$ z^{-q}H_+/z^qH_+$. Let $p_+: W/z^qH_+ \to H_+/z^qH_+$ be the orthogonal projection.
We say that  $W$ is of {\it virtual dimension zero}
if $\dim (\text{ker}\, p_+) = \dim (\text{coker}\, p_+),$
in other words if $\dim W/z^qH_+ = \dim H_+/z^qH_+ =q$.
 Denote $\Gr$ the set of all subspaces of
virtual dimension zero.

A subspace $W\in\GR$ has a basis $\{v_j\}_{j\geq 0}$ consisting of Laurent polynomials.
We may assume that the numbers $s_j = \ord\,v_j$
form a strictly  increasing sequence
$S_W =\{s_0<s_1<s_2<\dots\}$ and $v_j$ are just monomials for sufficiently large
 $j$. The assignment $W\mapsto S_W$ is well-defined. The  subset $S_W$ will be called the
 {\it order subset} of  $W$.

\begin{lem}
\label{lem strictly incr}
A subspace $W\in \GR$ lies in $\Gr$ if and only if the order subset $S_W$ is  of virtual cardinal zero, that is
$s_j=j$ for sufficiently large  $j$.
\qed
\end{lem}

\begin{lem}
\label{lem order empty}
If the order subset of  $W\in \GR$ is $S^\emptyset=\{0<1<\dots\}$, then $W=H_+$.
\qed
\end{lem}

For $W\in\Gr$ we may assume that $v_j = z^j$ for sufficiently large $j$.
We say that a basis $\{v_j\}_{j\geq 0}$ of $W\in\Gr$ is {\it special} if it
 consists of Laurent polynomials such that
 $v_j = z^j$ for sufficiently large $j$ and the numbers $s_j=\text{ord}\,v_j $ form a strictly increasing
sequence.

\subsection{Subspaces in $\Gr$ and subspaces in $\C[x]$}
\label{subs in all}

Let $W\in\Gr$. Let $S=\{s_0<s_1<\dots \}$ be the order subset of $W$.
Let $\la=(\la_0\geq \la_1\geq\dots)$ with
$\la_j = j - s_j$ be the corresponding partition. Let $n$ be such that $s_j=j$ for $j>n$  \ and hence
$\la_j=0$ for $j>n$. Let $\{v_j=\sum_{i\geq s_j}v_{j,i}z^i\}_{j\geq 0}$ be a special basis
of $W$ such that $v_j=z^j$ for $j> n$ and $v_{j,i}=0$ if $i>n$ and $j\leq n$.

Introduce the $n+1$-dimensional subspace $V_{W,n}\subset \C[x]$
as the subspace spanned by the polynomials $f_{j,n}(x), j=0,\dots,n$, given by the formula
\bean
\label{fP}
 f_{j,n}(x)= \sum_{i=0}^{\la_j+n-j} v_{j,n-i} \frac{x^i}{i!},
 \qquad
 j=0,\dots,n.
 \eean
We have $\deg f_j = \la_j+n-j$ for all $j$. The relation between $V_{W,n}$ and $V_{W,n+1}$
is given by the formula
\bean
\label{fni}
f_{j,n+1}(x) = \int_{0}^x f_{j,n}(u)du\quad\text{for}\quad
j=0,\dots,n, \quad \text{and}\quad f_{n+1,n+1}(x)=1 .
\eean
In other words, $V_{W,n+1}=\int V_{W,n} dx$.

Conversely, let $\la=(\la_0\geq\la_2\geq \dots)$ be a partition and
$S=\{s_0 <s_1<\dots\}$, with $s_j=\la_j-j$, the corresponding subset of virtual cardinal zero.
Let $n$ be such that $\la_j=0$ for $j>n$. Denote $\text{Gr}(\la,n,\C[x])$ the variety of all
$n+1$-dimensional subspaces $V\subset \C[x]$ such that $V$ has a basis consisting of  polynomials
 $f_j, j=0,\dots,n$, with $\deg f_j=\la_j+n-j$.

 Denote $P=\{\la_0+n,\la_1+n-1,\dots,\la_n\}$. Every $V\in \text{Gr}(\la,n,\C[x])$
 has a unique basis
\bean
\label{f poly}
f_j(x) = \sum_{i=0}^{\la_j+n-j} v_{j,n-i}\frac{x^i}{i!}, \qquad j=0,\dots,n,
\eean
where $v_{j,n-i}$ are some numbers such that $v_{j, j-\la_j}=1$ and $v_{j,n-i}=0$ if $n-i\in P-\{\la_j+n-j\}$.
The variety $\text{Gr}(\la,n,\C[x])$ is an affine space of dimension $|\la|$ with coordinates $\{v_{j,n-i}\}$.
The basis of $V\in \text{Gr}(\la,n,\C[x])$ given by \Ref{f poly} will be called {\it special}.

To every $V\in \text{Gr}(\la,n,\C[x])$ we assign $W_V\in \Gr$ with the basis $\{v_j\}_{j\geq 0}$, where
\bean
\label{bAs}
v_j=\sum_i v_{j,i}z^i, \qquad \text{for}\quad j=0,\dots,n,
\eean
$v_j=z^j$ for $j>n$.

\begin{lem}
\label{lem Sp sp}
If $W$ is as above, then $W_{V_{W,n}}=W$. If $V\in \text{Gr}(\la,n,\C[x])$, then the order set of $W_V$ is
$S=\{s_0<s_1<\dots\}$ with $s_j=j-\la_j$ and $V_{W_V, n}=V$.
\qed

\end{lem}

Let $\C_{|\la|}[x]$ be the affine space of monic polynomials of degree $|\la|$. It has dimension $|\la|$.
Define the Wronsky map
\bean
\label{mAP}
&&
\\
\notag
&&
\Wr_\la : \text{Gr}(\la,n,\C[x]) \to \C_{|\la|}[x], \quad
V \mapsto \frac 1{\prod_{0\leq i<j\leq n}(\la_j-\la_i+i-j)}\ \Wr(f_0,\dots,f_n),
\eean
where $f_0,\dots,f_n$ is the special basis of $V$. It is well-known that $\Wr_\la$ is a map
of finite degree $d_\la$, which is calculated in terms of Schubert calculus or representation theory,
see for example,  \cite{MTV3}, \cite{S}.

\subsection{Tau-functions}
\label{sec taU}

For a subspace $W \in \Gr$ with a special basis  $\{v_j\}_{j\geq 0}$ such that $v_j=z^j$ for $j>n$,
we  define the tau-function by the formula
\bean \label{eqn tau-definition}
\phantom{aaa}
\tau_{W}(t_1,t) = \text{det} \left(
\begin{matrix}
\sum_i v_{0,i} h_{-i} &
\sum_{i} v_{0,i} h_{-i + 1} &
\dots & \sum_{i} v_{0,i} h_{-i + n} \\
\sum_{i} v_{1,i} h_{-i} &
\sum_{i } v_{1,i} h_{-i + 1} &
\dots & \sum_{i} v_{1,i} h_{-i + n} \\
\dots & \dots & \dots & \dots \\
\sum_{i} v_{n,i} h_{-i} &
\sum_{i} v_{n,i} h_{-i + 1} &
\dots & \sum_{i} v_{n,i} h_{-i + n} \\
\end{matrix} \right),
\eean
see \cite{SW}, cf. \Ref{eqn wsum}.
The tau-function is a polynomial in a finite number of variables $t_1,t$.
For example,  $\tau_{H_+} = 1.$

The tau-function is independent of the choice of $n$ and
the choice of a special basis up to multiplication of the tau-function by a nonzero number.

\begin{example}
Let $S=\{s_0<s_1<\dots\}\subset \Z$ be a subset of virtual cardinal zero. Let $W_S$ be the subspace with
basis $\{z^{s_i}\}_{i\geq 0}$. Then the   tau-function $\tau_{W_S}$ of $W_S$ equals the Schur
polynomial $F_S$.

\end{example}

\begin{lem}
\label{lem wr tau}
We have
 \bean \label{eqn tau-wronskian}
\tau_W = \Wr_{t_1} \left(
\sum_{i } v_{i,0} h_{-i+n}
,\dots,
\sum_{i} v_{i,d} h_{-i+n}
\right) .
\eean
\qed
\end{lem}

\begin{lem}
\label{lem Wr tau}
Let $W$ be as in Lemma \ref{lem wr tau}.
Let the polynomials $f_0,\dots,f_n$ be given by formula \Ref{fP}. Then
\bean
\label{tau X}
\tau_W(t_1=x,t=0) = \Wr (f_0(-x),\dots,f_n(-x)).
\eean
\end{lem}

\begin{proof}
The lemma follows from Lemma \ref{lem wr tau} and formula \Ref{der h}.
\end{proof}

\begin{thm}
\label{finiteness}
Let $\la$ be a partition, $S$ the corresponding set of virtual cardinal zero,
$g(x)$
\linebreak
 a polynomial of degree $|\la|.$
Then the set of subspaces $W\in\Gr$, with order subset $S$ and such that
$\tau_W(t_1=x,t=0) = g(x)$, is finite. The number of such subspaces  (counted with multiplicities)
equals $d_\la$.
\qed
\end{thm}

\subsection{Properties of tau-functions}
\label{proP}

Let $S=\{s_0<s_1<\dots\}$,  $S'=\{s_0'<s_1'<\dots\}$ be subsets of $\Z$ of virtual cardinal zero.
We say that $S\leq S'$ if $s_j\leq s_j'$ for all $j$. Let $\la$, $\la'$ be the partitions
corresponding to $S$, $S'$, respectively.
 If $S\leq S'$ and $S\ne S'$, then $|\la|>|\la'|$.

\begin{lem}
\label{lem tau first}
Let a subspace  $W\in\Gr$ have order set $S=\{s_0<s_1<\dots\}$.
Then
\bean
\label{lead of tau}
\tau_W = \sum_{S'\geq S} w_{S'} F_{S'},
\eean
where the sum is over the subsets $S'$ of virtual cardinal zero
such that $S\leq S'$, $F_{S'}$ is the Schur polynomial associated with $S'$,
$w_{S'}$ is a suitable number, and $w_S\ne 0$.
\qed
\end{lem}

\begin{cor}
\label{cor on deg tau}
We have
\bean
\label{deg tau}
\tau_W\, =\, {a}\, t_1^{|\la|}\  +\ (\text{low order terms in }\, t_1 ),
\eean
where $a$ is a nonzero number independent of $t_1, t$.
\end{cor}

We define the {\it normalized tau-function} by the formula $\tilde\tau_W = \on{const}\,\tau_W$
where the const is chosen so that
\bean
\label{norm tau}
\tilde \tau_W  \, =\, t_1^{|\la|}\  +\ (\text{low order terms in }\, t_1 ).
\eean

\medskip

Let $S$ be a subset of virtual cardinal zero. Let $A=\{a_1,\dots,a_k\} \subset \Z$ be a finite
subset of distinct integers. Assume that  $\{a_1,\dots,a_k\} \cap (S+k) = \emptyset$.
Let $W\in\Gr$ be a subspace whose  order subset is $S$. Let $v_1,\dots,v_k$ be Laurent polynomials
such that $\ord \,v_i=a_i$ for $i=1,\dots,k$.

\begin{lem}
\label{lem shift subspace}
The subspace $z^kW \ + $\ \em span\em$\langle v_1,\dots,v_k\rangle$ is an element of\ $\Gr$ and
\linebreak
$\{a_1,\dots,a_k\} \cup (S+k)$ is its order subset.
\qed
\end{lem}

Let $S$ be a subset of virtual cardinal zero. Assume that the two element subset $A=\{a_1<a_2\}\subset \Z$ is such that
$A\cap (S+1)=\emptyset$.
Let $W\in\Gr$ be a subspace whose  order subset is $S$. Let $v_1,v_2$ be Laurent polynomials
such that $\ord \,v_i=a_i$ for $i=1,2$.
 Then the subspaces
$W_1  = zW \ + $\ \em span\em$\langle v_1\rangle$,
$W_2  = zW \ + $\ \em span\em$\langle v_2\rangle$
$W_3  = W$,
$W_4  = z^2W \ + $\ \em span\em$\langle zv_1, zv_2\rangle$
are elements of $\Gr$, by Lemma \ref {lem shift subspace}.

\begin{thm}
\label{thm new identity tau}
 We have
\bean
\label{formula new tau}
\Wr_{t_1}(\tau_{W_1}, \tau_{W_2})\,=\,\on{const}\, {} \ \tau_{W_3}\tau_{W_4} ,
\eean
where $\on{const}$ is a nonzero number independent of\  {} $t_1,t$.
\end{thm}

 \begin{proof}
 The proof follows from Lemma \ref{lem MV} and is similar to the proof of Theorem \ref{thm new identity}.
\end{proof}

Notice that Theorem \ref{thm new identity} is a particular case of Theorem \ref{thm new identity tau}.

Similarly to Section \ref{sec more general}, one can obtain more general Wronskian identities for tau-functions
from Lemmas \ref{MV lem1} and \ref{MV lem2}.

\begin{lem}
\label{lem inclusion}
Let $W_1,W_2\in \Gr$ and  $z^kW_1\subset W_2$ for some positive integer $k$.
Then $\dim W_2/z^kW_1 = k$.

\end{lem}

\begin{proof}
Since $z^kW_1\subset W_2$,  the order set $S_{W_2}$ contains the order set $S_{z^kW_1}$.
  Since $W_1,W_2\in \Gr$ the difference $S_{W_2}-S_{z^kW_1}$ consists of $k$ elements.
This proves the lemma.
\end{proof}

\subsection{KdV subspaces}
\label{Sec KdV subspaces}

Fix an integer $N>1$. We say that a subspace $W\in\Gr$ is a {\it KdV subspace} if
$z^NW \subset W$. For example, for any $N$ the subspace $H_+$
 is a KdV subspace.


\begin{lem}
\label{lem order KdV}
Let $W$ be a  KdV subspace with order subset $S$. Then $S$ is a KdV subset.
\qed
\end{lem}

Let $S$ be a KdV subset with leading term $A$. Recall the mutation KdV subsets
$S[a] = \{a+1-N\} \cup (S+1)$, $a\in A$.

\begin{lem}
\label{lem two step constr}
Let $W_1$ be a KdV subspace. Let $S_{W_1}$ be the order subset
of $W_1$ and  $A_1$  the leading term of $S_{W_1}$. Let $a\in A_1$ and let  $w\in W_1$ be
such that $\ord\, w=a$. Then $W_2=zW_1 + $\em span\em$\langle z^{1-N} w\rangle$ is a KdV
subspace with leading term $S_{W_1}[a]$ and
such that $zW_1\subset W_2$.
\qed

\end{lem}

\begin{lem}
\label{lem two step W}
Let $W_1,W_2$ be KdV subspaces such that $zW_1\subset W_2$. Let $S_{W_1}$ be the order subset
of $W_1$ and  $A_1$  the leading term of $S_{W_1}$. Then the order subset $S_{W_2}$ is a mutation
of $S_{W_1}$, that is $S_{W_2}=S_{W_1}[a]$ for some $a\in A$.
Moreover there exists $w\in W_1$ such that $\ord\,w=a$ and $W_2=zW_1 + $\em span\em$\langle z^{1-N} w\rangle$.
\qed

\end{lem}

\subsection{mKdV tuples of subspaces and mKdV flows, \cite{SW, W}}
\label{sec mkdv tuples of subspaces}

 We say that an $N$-tuple $\bs W = (W_1,\dots,W_N)$ of KdV subspaces is an {\it mKdV tuple of subspaces}
 if $zW_i\subset W_{i+1}$ for all $i$, in particular, $zW_N\subset W_1$.
 Denote $\GM$ the set of all mKdV tuples of subspaces.

For example, for any $N$
\bea
\bs W^\emptyset = (H_+,\dots,H_+)
\eea
is an mKdV tuple.

Following \cite{SW} consider the group $\Gamma_+$ of holomorphic maps $D_0\to\C^\times$, where
$D_0$ is the disc $\{z\in \C\ |\ |z|\leq 1\}$. The group $\Gamma_+$ acts on $\GM$ by multiplication operators
and induces  on $\GM$ commuting flows (called the mKdV flows) in the following sense: if $g =e^{\sum t_kz^k} \in \Gamma_+$, where
$t_1,t=(t_2,t_3,\dots)$ are numbers almost all zero, then $(gW_1,\dots,gW_N)$ is the mKdV tuple obtained from
$(W_1,\dots,W_N)$ by letting it flow for time $t_k$ along the $k$-th mKdV flow for each $k$.

These flows on $\GM$ project to the mKdV flows on Miura opers by the following construction.
Let $\bs W=(W_1,\dots,W_N)\in\GM$. Let $(\tau_{W_1},\dots,\tau_{W_N})$ be the tuple of the
corresponding tau-functions. If $g =e^{\sum t_k^0z^k} \in \Gamma_+$, then
$\tau_{gW_i}(t_1,t_2,\dots) = \tau_{W_i}(t_1+t_1^0,t_2+t_2^0,\dots)$ for all $i$, see \cite{SW}.
Define the Miura oper $\Ll_{\bs W} = \der + \Lambda + V$ by the formula
\bean
\label{Miura W}
&&
{}
\\
\notag
&&
V = \on{diag} \left( \log'\left(\frac{ \tau_{W_1}(x+t_1,t)}{ \tau_{W_N}(x+t_1,t)}\right),
\log'\left(\frac{ \tau_{W_2}(x+t_1,t)}{ \tau_{W_1}(x+t_1,t)}\right),\dots,
\log'\left(\frac{ \tau_{W_N}(x+t_1,t)}{ \tau_{W_{N-1}}(x+t_1,t)}\right)
\right).
\eean
This Miura oper depends on parameters $t_1,t$.

\begin{example}
Let $\bs S=(S_1,\dots,S_N)$ be an mKdV tuple of subsets. Then
$\bs W=(W_{S_1},\dots,W_{S_N})$ is an mKdV tuple of subspaces. For this $\bs W$
formula \Ref{Miura W} takes the form
\bea
V = \on{diag} \left( \log'\left(\frac{ F_{S_1}(x+t_1,t)}{ F_{S_N}(x+t_1,t)}\right),
\log'\left(\frac{ F_{S_{2}}(x+t_1,t)}{ F_{S_{1}}(x+t_1,t)}\right),\dots,
\log'\left(\frac{ F_{S_{N}}(x+t_1,t)}{ F_{S_{{N-1}}}(x+t_1,t)}\right)
\right).
\eea

\begin{thm} [\cite{W1}]
\label{thm Wilson}
For any $r\in\Z_{>0}$, the Miura oper $\Ll_{\bs W}$ satisfies the $r$-th mKdV equation, see
\Ref{mKdVr}.

\end{thm}

This theorem is Proposition 4.9 in \cite{W1}, see also formula (4.2) in \cite{W2} and its proof in \cite{W2}.

\end{example}

\subsection{Properties of mKdV tuples of subspaces}
\label{Properties of mKdV tuples of subspaces}

\begin{lem}
\label{lem permut tau}
If  $\bs W = (W_1,\dots,W_N)\in \GM$, then for any $i$, $(W_i,W_{i+1},\dots,W_N,$ $ W_1,W_2,\dots,W_{i-1})\in \GM$.
\qed
\end{lem}

\begin{lem}
\label{lem order tau}
Let  $\bs W = (W_1,\dots,W_N)\in\GM$. Let  $S_i$ be the order subset of $W_i$ and
  $\bs S = (S_1,\dots,S_N)$. Then $\bs S$ is an mKdV tuple of subsets.
\qed
\end{lem}

\medskip

Let $W$ be a KdV subspace with order subset $S$. Let $A=\{a_1<\dots<a_N\}$ be the leading term of $S$.
 Let $v=(v_1,\dots,v_N)$ be a tuple of elements of $ W$  such that
$\ord v_i=a_i$ for all $i$.
Let $\sigma$ be
an element of the permutation group $\Si_N$. Define an $N$-tuple $\bs W_{W,v,\sigma} = (W_1,\dots,W_N)$
of subspaces by the formula
\bean
\label{mkdv tuple sp}
W_i = \{z^{i-N}v_{\sigma(1)}, z^{i-N}v_{\sigma(2)}, \dots, z^{i-N}v_{\sigma(i)}\} + z^iW,
\eean
in particular, $W_N = z^NW \ + $\ \em span\em$\langle v_1,\dots,v_N\rangle = W$.

\begin{lem}
\label{thm mkdv descr tau}
The $N$-tuple $\bs W_{W,v,\sigma}$ is an mKdV tuple of subspaces.
Every  mKdV tuple of subspaces is of the form  $\bs W_{W,v,\sigma}$ for suitable $W,v,\sigma$.
\qed

\end{lem}

Here is another description of mKdV tuples of subspaces.

 \begin{thm}
 \label{thm mkdv spaces}

 Let $z^NW=V_0\subset V_1\subset V_2\subset \dots\subset V_{N-1}\subset V_N=W$
 be a complete flag of vector subspaces such that $\dim V_i/V_{i-1}=1$ for all $i$.
Set
\bean
\label{MKDV tuples}
W_i = z^{i-N}V_i,\qquad i=1,\dots,N.
\eean
Then  $\bs W = (W_1,\dots,W_{N-1}, W_N=W)$ is an mKdV tuple of subspaces. Every
 mKdV tuple of subspaces is of this form.
\qed

\end{thm}

Let $W$ be a KdV subspace.
It follows from Theorem \ref{thm mkdv spaces} that the set of mKdV tuples of subspaces with the prescribed last term $W_N=W$
is identified with the set of complete flags in  $W/z^NW$.

\subsection{Generation of new  mKdV tuples of subspaces}
\label{sec Gentions of  mKdV subspaces }

Let  $\bs W = (W_1,\dots,W_N)\in \GM$.  By Theorem \ref{thm mkdv spaces}, the tuple $\bs W$
 is determined by a flag $z^NW_N=V_0\subset V_1\subset V_2\subset \dots\subset V_{N-1}\subset W_N$.
The quotient $V_2/V_0$ is two-dimensional. Any line $\tilde V_1/V_0$ in $V_2/V_0$ determines a flag
$z^NW_N=V_0\subset \tilde V_1\subset V_2\subset \dots\subset V_{N-1}\subset W_N$ which in its turn
determines an mKdV tuple $\bs W^{(1)}=(\tilde W_1,W_2,\dots,W_N)$
with $\tilde W_1 = z^{1-N}\tilde V_1$.
Thus we get a family of mKdV tuples of subspaces parameterized by points
of  the projective line $P(V_2/V_0)$. The new tuples
are parametrized by points of the affine line $\A=P(V_2/V_0)-\{V_1/V_0\}$.
We get a map $X^{(1)}:\A\to \GM$ which sends   $a\in\A$ to the corresponding mKdV tuple
$\bs W^{(1)}(a)=(\tilde W_1(a),W_2,\dots,W_N)$.
This  map  will be called
the {\it generation of mKdV tuples from the tuple $\bs W$ in the first direction}.

Similarly, for any $i=1,\dots,N$, we construct a map $X^{(i)}:\A\to \GM$,
where $\A=P(V_{i+1}/V_{i-1}) - \{V_i/V_{i-1}\}$ which sends $a\in \A$
to the corresponding mKdV tuple
$\bs W^{(i)}(a)=(W_1,\dots,\tilde W_i(a),\dots,W_N)$.
This  map  will be called
the {\it generation of mKdV tuples of subspaces from the tuple $\bs W$ in the $i$-th direction.}

We will say that the generation in the $i$-th direction is {\it degree increasing} if
for any $a\in \A$, we have $\deg_{t_1} \tau_{\bs W^{(i)}(a)} > \deg_{t_1} \tau_{\bs W}$.

\medskip
The tau-function $\tau_{\tilde W_i(a)}$ depends on $a$ linearly in the following sense.
Let $\{v_i\}_{i\geq 1}$ be a basis of $V_{i-1}$. Let $v_0\in V_i$ be such that $\{v_i\}_{i\geq 0}$
is a basis of $V_i$. Let $\tilde v_{0}\in V_{i+1}$ be such that $\{\tilde v_{0}, v_0,v_1,v_2,\dots\}$ is a basis
of $V_{i+1}$. Then the points of $ \A=P(V_{i+1}/V_{i-1}) - \{V_i/V_{i-1}\}$ are parametrized by complex numbers $c$.
A  number $c$ corresponds to the line generated by the subspace $\tilde V_i(c)$ with basis $\{\tilde v_{0} + cv_0,
v_1,v_2\dots\}$. This $c$ is an affine coordinate on $\A$.  Calculating the tau-function of the  subspace
$\tilde W_i(c)=z^{i-N}\tilde V_i(c)$ with respect to the basis $\{z^{i-N}(\tilde v_{0} + cv_0),
z^{i-N}v_1, z^{i-N}v_2\dots\}$ we get the formula
\bean
\label{tau linear}
\tau_{\tilde W_i(c)}  = \tau_{\tilde W_i(0)} + c \tau_{W_i}.
\eean

\begin{thm}
\label{thm taU Wr}
For the generation in the $i$-th direction,
 the tau-functions of the subspaces $\tilde W_i(c), W_i,W_{i-1},W_{i+1}$ satisfy the
 equation
\bean
\label{wR tAu}
\Wr_{t_1}(\tau_{W_i}, \tau_{\tilde W_i(c)}) \,= \,\on{const}\, {} \ \tau_{W_{i-1}}\tau_{W_{i+1}},
\eean
where $\on{const}$ is a number independent of $t_1,t$.
\end{thm}

\begin{proof}
The proof follows from Lemma \ref{lem MV} and is similar to the proofs of Theorems \ref{thm new identity} and \ref{thm new identity tau}.
\end{proof}

\begin{cor}
\label{cor fertiL}
For any mKdV tuple $\bs W = (W_1,\dots,W_N)$ and any fixed $t$, the tuple
$(\tau_{W_1}(x,t),\dots,\tau_{W_N}(x,t))$
of polynomials in $x$ is fertile.
\qed
\end{cor}

\subsection{Normalized generation}
\label{sec norm gen}

 If  the generation in the $i$-th direction is degree increasing,  then the generation procedure
 can be normalized as follows.

 \begin{lem}
 \label{lem normal}
 Assume that the generation in the $i$-th direction is degree increasing. Then
 there exists a unique line  $\tilde V_{i,0}/V_{i-1} \in V_{i+1}/V_{i-1}$ such that
 the  subspace $\tilde W_{i,0}=z^{i-N}\tilde V_{i,0}$ has the following property. Consider
  the polynomial
 $\tilde\tau_{\tilde W_{i,0}} (t_1,t=0)$ in $t_1$. Then the monomial $t_1^{\deg_{t_1} \tilde \tau_{W_i}}$
 enters this polynomial
 with the zero coefficient.
 \qed
 \end{lem}

 Now choose  $\{v_i\}_{i\geq 1}$,  a basis of $V_{i-1}$. Choose $v_0\in V_i$ such that $\{v_i\}_{i\geq 0}$
 is a basis of $V_i$ and the tau-function of $W_i$ with respect to the basis
 $\{z^{i-N}v_i\}_{i\geq 0}$ is the normalized tau-function $\tilde \tau_{W_i}$.
 Choose $\tilde v_0\in V_{i+1}$ such that $\{\tilde v_0, v_i\}_{i\geq 1}$
 is a basis of $V_{i,0}$ and the tau-function of $\tilde W_{i,0}=z^{i-N}\tilde V_{i,0}$ with respect to the basis
 $\{z^{i-N}\tilde v_0,z^{i-N}v_i\}_{i\geq 1}$ is the normalized tau-function $\tilde \tau_{\tilde W_{i,0}}$.
Then we will parametrize the points of $ \A=P(V_{i+1}/V_{i-1}) - \{V_i/V_{i-1}\}$ by complex numbers $c$, where
a  number $c$ will correspond to the line generated by the subspace $\tilde V_i(c)$ with basis $\{\tilde v_0 + cv_0,
v_1,v_2\dots\}$.  Calculating the tau-function of the subspace
$\tilde W_i(c)=z^{i-N}\tilde V_i(c)$ with respect to the basis $\{z^{i-N}(\tilde v_{0} + cv_0),
z^{i-N}v_1, z^{i-N}v_2\dots\}$ we get the formula
\bean
\label{tau linear}
\tilde \tau_{\tilde W_i(c)}  = \tilde \tau_{\tilde W_{i,0}} + c \tilde \tau_{W_i}.
\eean
We get a map
\bean
\label{ti X}
 \tilde X^{(i)} : \C \to \GM, \qquad c \mapsto \bs W^{(i)}(c)=(W_1,\dots,\tilde W_i(c),\dots,W_N).
 \eean
The map $\tilde X^{(i)}$ is independent of the choice of vectors $v_0, \tilde v_{0}$ and the basis $\{ v_i\}_{i\geq 1}$.
This map will be called the {\it normalized generation in the $i$-th direction.}

\subsection{Multistep normalized generation in $\GM$}
\label{sec generation proced in GM}

Let $J = (j_1,\dots,j_m)$ be a degree increasing sequence of integers.
Starting from the mKdV tuple $\bs W^\emptyset=(H_+,\dots,H_+)$ and $J$, we will
construct by induction on $m$
a map
\bea
\tilde X^J : \C^m \to \GM.
\eea
If $J=\emptyset$, the map $\tilde X^\emptyset$ is the map $\C^0=(pt)\ \mapsto \bs W^\emptyset$.
If $m=1$ and $J=(j_1)$,  the map
$\tilde X^{(i_1)} :  \C \to \GM$ is given by formula \Ref{ti X}
for $\bs W=\bs W^\emptyset$ and $i=j_1$. In this case equations \Ref{wR tAu} and
\Ref{tau linear} take the form
\bea
\Wr_{t_1}(1, \tau_{ \tilde W_i(c)}) \,= \,\on{const},
\quad \text{and}\quad
\tilde \tau_{\tilde W_i(c)}  = t_1 + c.
\eea
Assume that for ${\tilde J} = (j_1,\dots,j_{m-1})$,  the map
$\tilde X^{{\tilde J}}$ is constructed. To obtain  $\tilde X^J$ we apply the normalized
generation procedure in the $j_m$-th
direction to every tuple of the image of $\tilde X^{{\tilde J}}$. More precisely, if
\bean
\label{J'}
\tilde X^{{\tilde J}}\ : \
{\tilde c}=(c_1,\dots,c_{m-1}) \ \mapsto \ (W_1({\tilde c}),\dots, W_N({\tilde c})).
\eean
Then
\bean
\label{J}
\tilde X^{J} : \C^m \mapsto \GM, \quad
({\tilde c},c_m) \mapsto
(W_1({\tilde c}),\dots,  \tilde W_{j_m}({\tilde c,c_m}),\dots,
W_N({\tilde c})),
\eean
where $\tilde W_{j_m}({\tilde c,c_m})$ is constructed as in Section \ref{sec norm gen}.
The map  $\tilde X^J$ will be called  the {\it normalized generation   from $ \bs W^\emptyset$ in the $J$-th direction}.

It follows from the construction, that
\bean
\label{induction tau}
\phantom{aaa}
\Wr_{t_1}(\tilde \tau_{W_{j_m}(\tilde c)}(t_1,t), \tilde \tau_{\tilde W_{j_m}(\tilde c, c_m)}(t_1,t))
 \,= \,\on{const}\,  \tilde\tau_{W_{{j_m}-1}(\tilde c)}(t_1,t)\, \tilde \tau_{W_{{j_m}+1}(\tilde c)}(t_1,t) ,
\eean
where const is a nonzero integer depending on $J$ only, and
\bean
\label{induction dependence}
\tilde \tau_{\tilde W_{j_m}(\tilde c, c_m)}(t_1,t) = \tilde \tau_{\tilde W_{j_m}(\tilde c, c_m=0)}(t_1,t) + c_m \tilde \tau_{W_{j_m}(\tilde c)}(t_1,t)).
\eean
For any $(\tilde c,c_m)\in\C^m$, consider the tuple
\bean
\label{tau-polynomials}
(\tilde \tau_{W_1(\tilde c)}(t_1,t),\dots,\tilde \tau_{\tilde W_i(\tilde c, c_m)}(t_1,t), \dots,
\tilde \tau_{W_N(\tilde c)}(t_1,t))
\eean
 of the normalized tau-functions of the tuple
 \bean
 \label{iMage}
 \tilde X^J(\tilde c,c_m)=(W_1({\tilde c}),\dots,  \tilde W_{j_m}({\tilde c,c_m}),\dots,
W_N({\tilde c}))   \in\GR.
 \eean
 Then each of the functions in \Ref{tau-polynomials} is a monic polynomial in $t_1$ and
the  degree vector of this tuple of polynomials in $t_1$
equals $\bs k^J$, see \Ref{gen vector}.

\subsection{Relations with the generation of critical points}
\label{sec relations}

In Section \ref{sec generation procedure} for any degree increasing sequence of integers
 $J = (j_1,\dots,j_m)$ we constructed a map
 \bean
 \label{Y map}
 Y^J : \C^m \to (\C[x])^N,
 \qquad
 c\mapsto (y_1(x,c),\dots,y_N(x,c)),
 \eean
   called  the  generation  of $N$-tuples
 from $ y^\emptyset$ in the $J$-th direction.
In Section \ref{sec generation proced in GM}
for any degree increasing sequence of integers $J$
we constructed a map
\bea
\tilde X^J : \C^m \to \GM,
\qquad
c\mapsto (W_1(c),\dots,W_N(c)),
\eea
 called  the  normalized generation   from $ \bs W^\emptyset$ in the $J$-th direction.
Introduce a map
\bean
\label{Ti tau}
\tilde \tau : \GM \to (\C[t_1,t])^N,
\qquad
(W_1,\dots,W_N) \mapsto
(\tilde\tau_{W_1}(t_1,t),\dots,\tilde \tau_{W_N}(t_1,t)) ,
\eean
which assigns the tuple of nomalized tau-functions to an mKdV tuple of subspaces.
We get the composed map
\bean
\tilde \tau\circ \tilde X^J : \C^m\to(\C[t_1,t)])^N,
\qquad
c \mapsto
(\tilde\tau_{W_1(c)}(t_1,t),\dots,\tilde \tau_{W_N(c)}(t_1,t)) .
\eean

\begin{lem}
\label{lem X=Y}
For any $c\in\C^m$, we have
\bean
\label{X=Y}
Y^J(c) = (\tilde\tau_{W_1(c)}(t_1=x,t=0),\dots,\tilde \tau_{W_N(c)}(t_1=x,t=0)) .
\eean
\qed
\end{lem}

Lemma \ref{lem X=Y} says that if a tuple $(y_1(x),\dots,y_N(x))$
represents a critical point which was generated from $y^\emptyset$ by a degree increasing generation, then there exists
an mKdV tuple $(W_1,\dots,W_N)$ such that $(y_1(x),\dots,y_N(x))$ equals the right hand side in \Ref{X=Y}.

\begin{lem}
\label{lem Uniqness}
The map
$\tilde X^J$ sends distinct points of $\C^m$ to distinct points of $\GM$.
\end{lem}

\begin{proof}
The lemma follows from Lemmas \ref{lem X=Y} and \ref{lem uniqeness}.
\end{proof}

\smallskip
The polynomials in the image of $\tilde \tau\circ \tilde X^J $ depend all together only
on  finitely many variables in $t=(t_2,t_3,\dots)$.
Denote them by $t_J=(t_2,t_3, \dots,t_{d+1})$. Then we get a map
\bean
\label{T map}
T^J : \C^{m+d} \to (\C[x])^N,
\qquad
(c,t_J) \mapsto (\tilde\tau_{W_1(c)}(x,t_J),\dots,\tilde \tau_{W_N(c)}(x,t_J)) .
\eean

\begin{thm}
\label{thm induction}
The map $T^J$ can be induced from the map $Y^J$ by a suitable epimorphic polynomial map
$f: \C^{m+d} \to\C^m$, that is $T^J= Y^J\circ f$.
\end{thm}

\begin{proof}
The proof is by induction on $m$ and follows from comparing equations \Ref{wronskian-critical eqn} and
\Ref{induction tau}, cf.  the proof of Theorem \ref{thm schur induced}.
\end{proof}

Recall the family of Miura opers $\mu^J :\C^m \to \mc M, \ c\mapsto \mu(Y^J(c))$, assigned to the map $Y^J$ in Section \ref{Miura opers with cr points}.
In Section \ref{Vector fields} we proved that the family
 $\mu^J$ is invariant with respect to all mKdV flows, see Corollary \ref{cor Main}.
Theorems \ref{thm Wilson} and \ref{thm induction} clearly give a new proof of this fact.

\subsection{Transitivity of the generation procedure}
\label{sec transit}

\begin{thm}
\label{thm exist}
Let $\bs W\in \GM$ be an mKdV tuple of subspaces.
 Then there exists a degree increasing sequence $J=(j_1,\dots,j_m)$ and a point
$c\in \C^m$ such that $\bs W=\tilde X^J(c)$.

\end{thm}

\begin{proof}
Let $\bs W=(W_1,\dots,W_N)$.
Let $\bs S=(S_1,\dots,S_N)$ be the tuple of the corresponding  order subsets.
The tuple $\bs S$ is an mKdV tuple of subsets by Lemma \ref{lem order tau}.
By Theorem \ref{thm mutation all}, there exists $j\in\{1,\dots,N\}$ such that
the mutation $w_j : \bs S \to \bs S^{(j)}=(S_1,\dots,\tilde S_j,\dots,S_N)$
is degree decreasing.
Let us consider the generation
\bea
X^{(j)}: A \to \GM, \qquad
a\mapsto (W_1,\dots,\tilde W_j(a),\dots,W_N),
\eea
 from $\bs W$ in the $j$-th direction.
By Corollary \ref{cor on deg tau}  and Theorem \ref{thm taU Wr}
there exists a unique $a\in \A$ such that $\deg_{t_1} \tau_{\tilde W_j(c)}<\deg_{t_1} \tau_{\tilde W_j}$.
It is clear that the generation from $X^{(j)}(a)$ in the $j$-th direction is degree increasing.
If the tuple $X^{(j)}(a)$ is not $\bs W^\emptyset$, then we may apply the same reason to
$X^{(j)}(a)$ and decrease the ${t_1}$-degree of another coordinate tau-function.
After finitely many steps of this procedure we will obtain  $\bs W^\emptyset$. The theorem is proved.
\end{proof}

\begin{thm}
\label{cr=mkdv}
If $\bs W= (W_1,\dots,W_n), \tilde {\bs W}=(\tilde W_1,\dots,\tilde W_n) \in \GM$ are such that
\bea
(\tilde\tau_{W_1}(t_1=x,t=0),\dots,\tilde \tau_{W_N}(t_1=x,t=0))
=(\tilde\tau_{\tilde W_1}(t_1=x,t=0),\dots,\tilde \tau_{\tilde W_N}(t_1=x,t=0)) ,
\eea
then $\bs W =\tilde{\bs W}$.

\end{thm}

\begin{proof}

By Theorem \ref{thm exist}, there exists a degree increasing sequence $J=(j_1,\dots,j_m)$ and a point
$c=(c',c_m)\in \C^m=\C^{m-1}\times\C$ such that $\bs W=\tilde X^J(c)$. The proof of Theorem \ref{cr=mkdv} is by induction on
$m$.

If $J=\emptyset$, then $\tilde\tau_{W_1}(t_1=x,t=0),\dots,\tilde \tau_{W_N}(t_1=x,t=0)) = (1,\dots,1)$.
This condition determines $\bs W$ uniquely, $\bs W=(H_+,\dots,H_+)$.

Assume that Theorem \ref{cr=mkdv} is proved for $m-1$.
Starting from $\tilde {\bs W}$ we generate in the
$j_m$-th direction the one-parameter
family $\hat {\bs W}(s)$ of tuples in $\GM$ as explained in Section \ref{sec Gentions of  mKdV subspaces }.
That family has exactly one value $s_0$ of the parameter $s$ such that
$\deg_{t_1} \tilde \tau_{\hat {W}_{j_m}(s_0)}< \deg_{t_1} \tilde \tau_{\tilde {W}_{j_m}}$.
By induction assumption, we have  $\hat {\bs W}(s_0)= \tilde X^{J'}(c')$, where $J'=(j_1,\dots,j_{m-1})$.
Therefore, $\bs W =\tilde{\bs W}$.
\end{proof}

\begin{cor}
\label{cor last}
The  points of the set $\GM$ of mKdV tuples are in one-to-one correspondence with the tuples $(y_1,\dots,y_n)$ of the population of tuples
 generated  from $y^\emptyset$, see Section \ref{sec generation procedure}. The correspondence is
 \bean
(W_1,\dots,W_n)\mapsto (\tilde\tau_{W_1}(t_1=x,t=0),\dots,\tilde \tau_{W_N}(t_1=x,t=0)) .
\eean
\qed
\end{cor}


\bigskip
\noindent
{\bf Acknowledgments.}

\medskip
The authors thank B. Dubrovin, E. Frenkel, E. Mukhin, A. Lascoux, V. Schechtman, R. Stanley, G. Wilson, A. Zelevinsky for useful discussions.
The first author thanks the Hausdorff Institute and  IHES for hospitality. The second author thanks
the Hausdorff Institute for hospitality.  This work was supported in part by the National Science
Foundation:  NSF grant DMS--1101508.

\end{document}